\documentclass[11pt]{article}
\usepackage[scale=.8]{geometry}
\usepackage{graphicx}
\usepackage{amssymb}
\usepackage{amsmath}
\usepackage{hyperref}
\usepackage{enumerate}
\usepackage{amsthm}
\usepackage{amsmath}
\usepackage{float}
\usepackage{mathtools}
\usepackage{lineno}
\usepackage{bbm}
\usepackage{textcomp}
\usepackage{esint}
\usepackage{cleveref}
\usepackage{extarrows}

\usepackage{xcolor}
\definecolor{lanse}{RGB}{0,0,255} 
\definecolor{zise}{RGB}{112,48,160} 
\definecolor{hongse}{RGB}{200,0,0} 
\usepackage{lineno}
\makeatletter
\renewenvironment{proof}[1][\proofname]{%
	\par\pushQED{\qed}\normalfont%
	\topsep6\p@\@plus6\p@\relax
	\trivlist\item[\hskip\labelsep\bfseries#1\@addpunct{.}]%
	\ignorespaces
}{%
	\popQED\endtrivlist\@endpefalse
}
\makeatother

\theoremstyle{thmstyleone}%
\newtheorem{theorem}{Theorem}%  meant for continuous numbers
\newtheorem{proposition}[theorem]{Proposition}% 

\theoremstyle{thmstyletwo}%
\newtheorem{remark}{Remark}%

\theoremstyle{thmstylethree}%
\newtheorem{definition}{Definition}%

\theoremstyle{thmstylefour}%
\newtheorem{lemma}{Lemma}

\theoremstyle{thmstylefive}%
\newtheorem{assumption}{Assumption}

\theoremstyle{thmstylesix}%
\newtheorem{conjecture}{Conjecture}

\makeatletter
\def\keywords{\xdef\@thefnmark{}\@footnotetext}
\makeatother

\renewcommand{\d}{\mathrm{d}}
\newcommand{\dcon}{\Longrightarrow}
\newcommand{\Q}{\mathbb{Q}}
\newcommand{\R}{\mathbb{R}}
\newcommand{\F}{\mathbb{F}}
\newcommand{\N}{\mathbb{N}}
\newcommand{\E}{\mathbb{E}}
\newcommand{\ND}{\mathcal{N}}
\newcommand{\T}{\mathcal{T}}
\renewcommand{\P}{\mathbb{P}}
\newcommand{\ind}[1]{\mathbbm{1}\left\{ #1\right\}}
\newcommand{\abs}[1]{\lfloor #1 \rfloor}

\begin{document}

	\keywords{ This work was supported in part by  NSFC (No.~11971062), and the National Key Research and Development Program of China (No.~2020YFA0712900). }

	\author{
		Wenxin Fu$^*$ \quad Wenming Hong$^\dagger$ 
	}
	\title{On the maximal displacement of critical branching random walk in random environment}

	\date{{\small  {\it  
				$^{*}$$^{\dagger}$School of Mathematical Sciences \& Laboratory of Mathematics and Complex Systems, \\Beijing, China\\
				$^{*}$E-mail:  {\tt 202431130073@mail.bnu.edu.com} \\
				$^\dagger$E-mail:  {\tt wmhong@bnu.edu.com}
			}
		}
	}
	
	\maketitle
	\begin{abstract}
		In this article, we study the maximal displacement of critical branching random walk in random environment. Let $M_n$ be the maximal displacement of a particle in generation $n$, and $Z_n$ be the total population in generation $n$, $M$ be the rightmost point ever reached by the branching random walk. 
		Under some reasonable conditions,  we prove a conditional limit theorem,
		\begin{equation*}
			\mathcal{L}\left( \dfrac{M_n}{\sqrt{\sigma} n^{\frac{3}{4}}} |Z_n>0\right) \dcon \mathcal{L}\left(A_\Lambda\right),
		\end{equation*}
		where
		random variable $A_\Lambda$ is related to the standard Brownian meander. And there exist some positive constant $C_1$ and $C_2$, such that
		\begin{equation*}
			C_1\leqslant\liminf\limits_{x\rightarrow\infty}x^{\frac{2}{3}}\P(M>x) \leqslant 
			\limsup\limits_{x\rightarrow\infty} x^{\frac{2}{3}}\P(M>x) \leqslant C_2.
		\end{equation*}
		Compared with the constant environment case (Lalley and Shao (2015)), 
		it revaels that, the conditional limit speed for $M_n$ in random environment (i.e., $n^{\frac{3}{4}}$) is significantly greater than that of constant environment case (i.e., $n^{\frac{1}{2}}$), and so is the tail probability for the $M$ (i.e., $x^{-\frac{2}{3}}$ vs $x^{-2}$).
		Our method is based on the path large deviation for the reduced critical branching random walk in random environment. 
	\end{abstract}
	
	\section{Introduction}\label{sec1}

	Spatial branching systems have been extensively investigated over past
	decades, in which the maxima of the $n$-th generation of the branching random walk is one of the keynotes.
	%For time homogeneous branching random walk,we always distinct among three subclasses according to the mean number of offspring, which we denote by $m$.
	For supercritical branching random walk ($m > 1$),
	%the work on 
	the law of large numbers for the maxima of branching random walk in generation $n$, $M_n$,  can trace back to Hammersley \cite{Hammersley},
	Kingman \cite{Kingman}, Biggins \cite{Biggins1976} and Bramson \cite{Bramson}.
	%Besides, the position of rightmost particle
	% at a specific generation for supercritical branching random walk was
	Since then extensively studied on this topics have appeared in recent years, 
	see for example
	\cite{Addario2009,Aidekon2013,Bachmann,Bramson2016,Bramson2009,Hu2009} 
	and references therein.
	In particular A\"{\i}d\'{e}kon proved in \cite{Aidekon2013} 
	that  the distribution of the  centered maximal  
	converges in law to a random shift of the Gumbel distribution
	(see also \cite{Bramson2016}).
	
	For critical cases, the system will die out eventually, 
	the maximal displacement of the system is finite almost surely, and it is natural to consider the tail probability of the maximal displacement.
	The asymptotic law for the maxima of a critical branching Brownian
	motion trace back to Sawyer and Fleischman \cite{Fleischman1979Maximun}
	and Lalley and Sellke \cite{Lalley1987Condition}.
	While for critical  branching random walk, results appeared in recent years.
	The case when the offspring distribution is critical, 
	that is $m = 1$, was considered by Kesten \cite{Kesten1995}  
	and Lalley and Shao in \cite{Lalley2015}.
	Let $M_n$ be the maximal displacement in generation $n$, and $Z_n$ be the total population in generation $n$. 
	By introducing the discrete Feynman-Kac formula,
	Lalley and Shao proved in \cite{Lalley2015},
	under some moment assumptions (Theorem 3 in  \cite{Lalley2015}),
	\begin{equation}\label{cl1}
		\mathcal{L}\left( \dfrac{M_n}{ n^{\frac{1}{2}}} |Z_n>0\right) \dcon G, \quad \text{ as\quad  $n\to \infty$,}
	\end{equation}
	where $G$ is a  nontrivial  distribution  that depends only on the variances  of the offspring and step distributions. For 
	$M$, the rightmost point ever reached by the branching random walk,  (Theorem 1 in \cite{Lalley2015}),
	\begin{equation}\label{pt1}
		P(M\geqslant x) \sim \dfrac{\alpha}{x^2},\quad\text{as}\quad x\rightarrow\infty.
	\end{equation}
	Here $\alpha$ is a constant
	which depends on the standard deviations of
	the jump and offspring distribution.

	\
	
	Consider a branching random walk in random
	environment (BRWre), i.e., a branching random walk with the time-inhomogeneous environment, which has been introduced by Biggins and Kyprianou in \cite{Biggins2004}. For the maximal displacement in generation $n$, $M_n$, of 
	the supercritical branching random walk in random environment, Huang and
	Liu \cite{Huang2014} proved that the maximal displacement in the process grows at ballistic speed almost surely.
	Mallein and Mi\l o\'s investigated the second order behavior  in \cite{Mallein2019}, i.e.,
	\begin{equation}\label{stp}
		\lim\limits_{n\rightarrow\infty}\dfrac{M_n-\frac{K_n}{\theta^*}}{\log n} = -\varphi
	\end{equation}
	in  probability with respect to the annealed law
	(i.e., averaging over the branching random walk and point process laws).
	Here, $\theta^*$and $\varphi$ 
	are deterministic constants for which descriptions are given in \cite{Mallein2019}, and $K_n$ is an environment-measurable random walk. The random environment make effects on both  the speed of $M_n$, 
	that is the limit (in probability) of $M_n$, 
	is strictly greater than that seen in the time-homogeneous case and the logarithmic correction 
	is also strictly greater than in the time-homogeneous case (see \cite{Mallein2019}).

	In this article, we focus on the critical branching random walk in random environment.
	Different from the method `` discrete Feynman-Kac formula" by Lalley and Shao   in \cite{Lalley2015}, we apply the path large deviation method to prove (\Cref{Conditional limit theorem}), for $M_n$, conditional on survival events (i.e. $\{Z_n>0\}$),  the annealed law of the maximal displacement in generation $n$ convergence weakly to some non-degenerated random variable $A_\Lambda$, i.e.
	\begin{equation}
		\mathcal{L}\left( \dfrac{M_n}{\sqrt{\sigma} n^{\frac{3}{4}}} |Z_n>0\right) \dcon \mathcal{L}\left(A_\Lambda\right), \quad \text{as} \quad n\rightarrow\infty.
	\end{equation}
	For 
	$M$, the rightmost position ever reached by the branching random walk, we prove that (\Cref{Tail prob of M}),
	\begin{equation}\label{pt2}
		\P(M\geqslant x) \asymp \dfrac{\alpha}{x^{2/3}},\quad\text{as}\quad x\rightarrow\infty.
	\end{equation}
	under the annealed probability $\P$.
	
	\

	To formulate  precisely,
	let $\mathcal{P}\left(\N_0\right)$ 
	be the space of probability measures on 
	$N_0:=\left\{0,1,\cdots\right\}$.
	Equipped with the metric of total variation,
	$\mathcal{P}\left(\N_0\right)$ becomes a Polish space. 
	For any probability measure $F\in \mathcal{P}\left(\N_0\right)$, we also use $F$ to denote the generating function of this probability measure on $\N$, the mean value and normalized second factorial moment of $F$ is denoted as
	\begin{equation}\label{mean and stand-var}
		\overline{F}:= \sum_{k=0}^{\infty} k F[k],
		\quad\tilde{F}:=\dfrac{1}{\overline{F}^2} \sum_{k=1}^{\infty}k(k-1)F[k].
	\end{equation}
	Also, denote by $ \varkappa(F,a) $ 
	the standardized truncated second moment of 
	the probability measure $F\in \mathcal{P}\left(\N_0\right)$,
	\begin{equation*}
		\varkappa(F,a):= \dfrac{1}{\overline{F}^2} \sum_{y=a}^{\infty} y^2F[y].
	\end{equation*}
	
	Let $F$ be a random variable 
	taking values in $\mathcal{P}\left(N_0\right)$.
	Then an infinite sequence $\xi:= \left\{F_1,F_2,\cdots\right\} $
	of i.i.d. copies of $F$ is said to form a random environment, 
	we use $P$ denote the corresponding probability of environment.

	Given an environment $\xi:= \{F_1,F_2,\cdots\}$, 
	the time-inhomogeneous branching random walk 
	in environment $\xi$ is a process constructed as follows:
	\begin{itemize}
		\item It starts with one individual located at the origin in time $0$.
		\item At time $n$ ($n\geqslant1$), each individual alive at time $n-1$ dies and reproduces several children according to the probability measure $F_n$.
		\item Every new child 
		moves independently according to some jump distribution $\mu$
		respect to its parent (and independent of all other factors).
	\end{itemize}
	
	We denote by $\T $ the (random) genealogical tree of the process. 
	For a given individual $u\in \T$, we use $V(u)\in \R$ for the position of $u$ and $|u|$ for the generation where $u$ is alive.
	The pair $\left(\T,V\right)$ is called the branching random walk 
	in the time-inhomogeneous environment $\xi$.
	For each $n\geqslant0$, define $Z_n:= \# \{u\in \T,|u|=n\}$ the number of particles survived in generation $n$.
	
	Given the environment $\xi$, 
	use $P_\xi$ denote the quenched probability of 
	the branching random walk $\left(\T,V\right)$.
	And use $\P=P\otimes P_{\xi}$ 
	denote the annealed probability of 
	the branching random walk in random environment.
	Define
	\begin{equation}\label{quenched random walk}
		X_k:=\log \overline{F}_k,\quad \eta_k:=\tilde{F}_k,
		\quad S_0=0,\quad \text{and} \quad S_n=S_{n-1}+X_n.
	\end{equation}
	Under $\P$, $\{X_k;k\geqslant1\}$ is 
	a sequence of $i.i.d$ copies of logarithmic of the mean of the offspring number $X:=\log\overline{F} $,
	and the sequence $ S:=\left(S_0,S_1,\cdots\right) $ 
	is called the associated random walk.
	
	\
	
	\begin{assumption}\label{A1}
		Assume,
		
		(1) For the offspring
		\begin{equation*}
			\E[X]= 0, \quad\sigma^2:=\E[X^2]\in (0,\infty).
		\end{equation*}
		and there exist positive integer $a$ and some positive number $\delta>0$, such that
		\begin{equation*}
			\E\left[ \left(\log^+ \varkappa(F,a)  \right)^{2+\delta} \right] <\infty.
		\end{equation*}
		Here $\log_+x:= \log\left(\max\{x,1\}\right)$. 
		
		(2) The jump distribution $\mu$ is standard normal distribution $\ND(0,1)$.
	\end{assumption}
	
	\
	
	\begin{remark}
		
		(1)	Under the assumption, the branching processes in random environment will extinct eventually, actually, there exist some positive and finite constant $K$, such that
		\begin{equation}\label{Extinction Prob Decay}
			\lim\limits_{n\rightarrow\infty}\sqrt{n}\P\left(Z_n>0\right)=K.
		\end{equation}
		This was first proved by Kozlov	(\cite{Kozlov1976}) for  linear fractional 
		critical branching processes in $i.i.d.$ environments, and then, Afanasyev  et.al (\cite{Afanasyev2005})  extended  to more general cases as for \Cref{A1}.
		
		(2)For simplicity, the jump distribution we consider is the standard normal distribution $\ND(0,1)$. In fact, it is reasonable for any distribution that satisfies the Cram$\acute{e}$r's condition such that we can deal with the path large deviation for the branching random walk. This is the usually condition for the supercritical branching random walk. Actually, conditioned on non-extinction, we will consider the reduced critical branching random walk in random environment which behaves more like time-inhomogeneous supercritical case.
		
		%(i.e. the Laplace transform of the distribution is finite for some ball of $0$). 
		%For more general distribution without Cram$\acute{e}$r's condition, our method may don't work.
	\end{remark}

	According to \cref{Extinction Prob Decay}, the branching system will die out eventually, that means the maximal displacement of the branching system is finite almost surely.
	In this article, we will investigate the maximal displacement in generation $n$ condition on survival events (i.e. $\{Z_n>0\}$),
	and the whole maximal displacement of this system $\left(\T,X\right)$, i.e.
	\begin{equation*}
		M_n:= \max\{X(u):u\in\T,|u|=n\}, \quad M:= \sup M_n.
	\end{equation*}
	
	Conditioned on survival events $ \{Z_n>0\} $, 
	we will consider the maximal displacement in generation $n$ (i.e. $M_n$) firstly, and obtain the following  conditional limit theorem.
	
	\
	
	\begin{theorem}\label{Conditional limit theorem}
		Under the \Cref{A1}, 
		condition on the survival events $\left\{Z_n>0\right\}$,
		$\dfrac{M_n}{\sqrt{\sigma}n^{\frac{3}{4}}}$ convergence in law to some non-degenerated random variable $A_\Lambda$:
		\begin{equation}\label{cl2}
			\mathcal{L}\left( \dfrac{M_n}{\sqrt{\sigma} n^{\frac{3}{4}}} |Z_n>0\right) \dcon \mathcal{L}(A_\Lambda),\quad \text{as} \quad n\rightarrow\infty.
		\end{equation}
		Here $ A_\Lambda:=\sup\{g(1): \forall r\in[0,1],\int_{0}^{r} \frac{1}{2}g^{\prime}(s)^2 \d s\leqslant \Lambda_r,g\in C_0([0,1])\} $, where $ \Lambda_t:=\inf\{W_s^+;s\in[t,1]\} $ is the minimal process of a standard Brownian meander $\{W_t^+;t\in [0,1]\}$ and $C_0([0,1])$ is the set of continuous functions on interval $[0,1]$ with value $0$ at $0$.
	\end{theorem}
	
	\
	
	For the whole maximal displacement (i.e. $M$) of 
	critical branching random walk in random environment, we prove the asymptotic order of the  tail   probability.
	\begin{theorem}\label{Tail prob of M}
		Under the \Cref{A1}, there exist some finite positive constant $C_1$ and $C_2$, such that
		\begin{equation}\label{tail estimate}
			C_1\leqslant\liminf\limits_{x\rightarrow\infty}x^{\frac{2}{3}}\P(M>x) 
			\leqslant 
			\limsup\limits_{x\rightarrow\infty} x^{\frac{2}{3}}\P(M>x) \leqslant C_2
		\end{equation}
	\end{theorem}
	
	\
	
	\begin{conjecture}\label{Conj1}
		The constant $C_1$ and $C_2$ in \Cref{Tail prob of M} are equal.
	\end{conjecture}

	Compared with the behavior of $M_n$ in \cref{cl1} and \cref{cl2}, 
	it reveals that in the critical case, the conditional limit speed for $M_n$ in random environment (i.e., $n^{\frac{3}{4}}$) is significantly greater than that of constant environment ( time-homogeneous) case (i.e., $n^{\frac{1}{2}}$), and so is the tail probability for the $M$ (see \cref{pt1} and \cref{tail estimate}). The reason behind the phenomena is that, by Yaglom theorem for the critical branching processes, conditioned on non-extinction, the number of the particles $Z_n$ in generation $n$ is $Z_n\sim n$ for the constant environment and $Z_n\sim e^{c n^{1/2}}$ for the random environment case. As a consequence, the significantly larger number of living particles (in the random environment) makes it possible to jump to a higher position even though we assume the step jump is light tail (satisfying the Cram$\acute{e}$r's condition condition, see  \Cref{A1}). Actually, conditioned on non-extinction, we need only to consider the reduced critical branching random walk in random environment which behaves more similar to the time-inhomogeneous supercritical case, which enable us to prove the conditional limit by the path large deviation and then get the tail behavior for the $M$, the rightmost point ever reached by the branching random walk. 
	
	\
	
	The rest of the paper is as follows. In \Cref{sec2}, we introduce a kind of time-inhomogeneous branching random walk. And we will establish the sample path deviation principle \Cref{Moderate Deviation} under \Cref{A2} in \Cref{Sec2.1}. Based on the deviation principle, we will investigate the maximal displacement in generation $n$ (i.e. \Cref{Max Dis of TIBRW}) under \Cref{A2} and \Cref{A5}, which will play an important role in proving \Cref{Conditional limit theorem}. 
	
	In \Cref{Sec3}, we consider the reduced critical branching processes in random environment, and proved a stronger conditional limit theorem, which ensures the conditional reduced critical branching random walk in random environment satisfies \Cref{A2} and \Cref{A5} in \Cref{sec2}, thus \Cref{Max Dis of TIBRW} can be applied.
	
	Finally, with the tools established in \Cref{sec2} and \Cref{Sec3} in hand, we will prove \Cref{Conditional limit theorem} in \Cref{Sec4.1} and  prove \Cref{Tail prob of M} in \Cref{Section4.2}.

	\section{Time-inhomogeneous Branching Random Walk}\label{sec2}

	For the time-inhomogeneous branching random walk,
	in \cite{Fang2012}, Fang and Zeitouni study the maximal displacement of branching random walks in a class of time inhomogeneous environments. In the main results, an interesting phenomenon appears about the asymptotic behavior of the maximum displacement: the profile of the variance matters, both to the leading (velocity) term and to the logarithmic correction term, and the latter exhibits a phase transition.
	In \cite{Mallein2015(2)}, Mallein consider more general time-inhomogeneous branching random walk, and obtain similar results that is related to a optimization problem.
	
	In \cite{Fang2012(2)}, Fang and Zeitouni investigate a kind of time-inhomogeneous branching Brownian motion, obtain the leading (velocity) term and specify  the correction term being of order $n^{\frac{1}{3}}$; later in \cite{Mallard2016}, Maillard and Zeitouni refined the results: set $m_t= v(1)t-w(1)n^{\frac{1}{3}} -\sigma(1)\log n $ and some suitable $m_t^*$ satisfies  $\sup|m^*_t-m_t|<\infty$, where $w(1)$ is related to the largest zero of the Airy function of the first kind, the tightness about $M_t-m_t$, the convergence in law about $\{M_t-m^*_t\}$ and the limit distribution was obtained. And about time-inhomogeneous branching random walk, Mallein proved similar results in \cite{Mallein2015}.
	
	\
	
	In present paper, we  deal with the $M_n$, the maximal displacement in generation $n$ of the branching random walk in random environment. It is reasonable to realize that only the particles alive in generation $n$ make contributions to $M_n$. So conditioned on non-extinction, we need only to consider the so called reduced critical branching random walk in random environment, which behaves more like time-inhomogeneous supercritical case. 
	
	To this end, in this section we  will introduce a more general framework of the time-inhomogeneous branching random walk and establish the sample path moderate deviation principle, and then prove the so-called the law of large number of the maximal displacement $M_n$ of time-inhomogeneous branching random walk under some conditions.

	\subsection{Model and Assumptions}\label{Sec2.1}
	
	For each $n$, $\F^{n} := \{F_1^{n},F_2^{n},\cdots,F_n^{n}\}\in \mathcal{P}\left(\N_0\right)^n$ is a $n$ length sequence of branching mechanism.
	Given $ \F^{n} $, we can construct a time-inhomogeneous branching random walk up to time $n$, whose jump distribution is $\mathcal{N}(0,1)$ and independent of the branching, we use $\left(\T^n,V^n\right)$ denote this branching system. $N_k^n$ denote the set of particles alive at time $k$, and $Z_k^n:=|N_k^n|$ be the number of particles lives in generation $k$.  $P_n$ and $E_n$ denote the probability and the expectation of the time-inhomogeneous branching random walk.
	
	\
	
	\begin{assumption}\label{A2}
		For each $n$ and $1\leqslant k\leqslant n$, we assume that $F_k^{n}[0]=0$, then the mean of $F_k^n$ satisfies $\overline{F}^n_k\geqslant 1$.
		Let $S_0^n:=0 $ and $S^{n}_k := \sum\limits_{i=1}^{k} X_i^{n}$ for $1\leqslant k\leqslant n$, where $X_i^{n} := \log \overline{F}_i^{n} \geqslant 0$. Assume there exist a sequence finite positive number $\{a_n\}$ such that $\lim\limits_{n\rightarrow\infty} a_n = \infty$, $\lim\limits_{n\rightarrow\infty} \frac{a_n}{n} = 0$ and
		\begin{equation}\label{Techanical assumption}
			\lim\limits_{n\rightarrow\infty}
			\dfrac{1}{a_n}\log\left(1+\sum_{k=1}^n \tilde{F}^n_k\right)=0. 
		\end{equation}
		Assume $f(t)$ being some continuous non-decreasing function on $[0,1]$, satisfies $f(0)=0$ and $f(t)>0$ for $t>0$.
		For any $t\in [0,1]$,
		\begin{equation*}
			\lim\limits_{n\rightarrow\infty}\dfrac{S^n_{\abs{nt}}}{a_n} = f(t).
		\end{equation*}
	\end{assumption}

	\begin{remark}\label{Remark an}
		The asymptotic relationship satisfied by the sequence $\{a_n\}$ appeared in \Cref{A2} is consistent with the moderate deviation of the random walk with Cram$\acute{e}$r's condition, see Theorem 3.7.1 in \cite{LDP1998}. And for the normal distribution cases, the condition $\lim\limits_{n\rightarrow\infty} \frac{a_n}{n} = 0$ can be omitted.
	\end{remark}
	
	\
	
	\begin{lemma}\label{Remark of Convergence}
		Under \Cref{A2}, 
		$ \left\{\dfrac{S^n_{\abs{n\cdot}}}{a_n};n\geqslant1\right\} $  converges to $f(\cdot)$ uniformly.
	\end{lemma}
	
	\begin{proof}[Proof of \Cref{Remark of Convergence}]
		Define the continuous path
		\begin{equation*}
			f_n(t):=
			\begin{cases}
				\dfrac{S^n_{\lfloor nt\rfloor}}{a_n}  + \dfrac{1}{a_n} 
				\left(nt-\lfloor nt\rfloor\right) X_{\lfloor nt\rfloor +1},\quad &t\in[0,1)
				\\
				\dfrac{S_n^n}{a_n},\quad &t=1.
			\end{cases}
		\end{equation*}
		
		Under \Cref{A2}, it is not hard to find the increasing functions $f_n(t)$ also convergence to $f(t)$ in $[0,1]$.
		For any $\epsilon >0$, due to the continuous of $f$, we know there exist $\delta>0$, such for any $|x-y|<\delta$, then $|f(x)-f(y)|<\epsilon$.
		Choose some $N$ large enough, such $ \dfrac{1}{N} <\epsilon $,
		then for any $0\leqslant k\leqslant N-1$, 
		\[
		\left| f(\frac{k}{N})-f(\frac{k+1}{N}) \right|<\epsilon.
		\]
		For each $0\leqslant k\leqslant N$, as $f_n(\dfrac{k}{N}) \rightarrow f(\dfrac{k}{N})$, so there exist $C(k)$ large enough such that for any $n>C(k)$, 
		\[
		\left| f_n(\dfrac{k}{N})-f(\dfrac{k}{N}) \right|<\epsilon.
		\]
		Let $C:= \max\{C(k):0\leqslant k\leqslant N\}$, then for any $n>C$ and all $x\in [0,1)$,
		\begin{align*}
			|f_n(x)-f(x)|
			&\leqslant \left|f(x)-f_n(\dfrac{\lfloor Nx\rfloor}{N} ) \right|
			+
			\left| f_n(\dfrac{\lfloor Nx\rfloor}{N} ) -f(\dfrac{\lfloor Nx\rfloor}{N} ) \right|
			+\left| f(\dfrac{\lfloor Nx\rfloor}{N} ) -f(x)\right|
			\\
			&<
			f_n(\dfrac{\lfloor Nx\rfloor+1}{N})-f_n(\dfrac{\lfloor Nx\rfloor}{N} )
			+\epsilon
			+\epsilon
			\\
			&\leqslant
			f(\dfrac{\lfloor Nx\rfloor+1}{N})+\epsilon-
			\left(f(\dfrac{\lfloor Nx\rfloor}{N} )-\epsilon\right)
			+2\epsilon
			\\
			&=
			f(\dfrac{\lfloor Nx\rfloor+1}{N} )-f(\dfrac{\lfloor Nx\rfloor}{N} )
			+4\epsilon
			<5\epsilon.
		\end{align*}
		Thus the proof is over.
	\end{proof}
	
	\begin{assumption}\label{A5}
		Under \Cref{A2}, for the time-inhomogeneous branching processes, for any $\varepsilon>0$, and each $t\in (0,1]$, 
		\begin{equation*}
			\lim\limits_{n\rightarrow\infty}
			P_n\left(\left|\dfrac{\log Z_{\abs{nt}}^n}{a_n}-f(t)\right|>\varepsilon\right)=0
		\end{equation*}
	\end{assumption}
	
	\begin{remark}
		About the \cref{Techanical assumption} in \Cref{A2}, it is a technical assumption.
		What's more, if all the branching mechanisms belong to the linear fractional case, \cref{Techanical assumption} can result to the \Cref{A5} directly (just calculate the generating function). 
		However, for more general cases, although \Cref{A2} are satisfied, \Cref{A5} may be false.
		
		For classical supercritical branching processes (non-extinction), when the $L\log L$ condition is satisfied, the martingale $\{W_n:=\frac{Z_n}{m^n}\}$ uniformly convergence to some non-negative random variable $W$, and satisfies $\P(W>0)=\P\left(\{\text{survival forever}\}\right)=1$ which means the \Cref{A5} is right. However, for general time-inhomogeneous branching process, we don't have such results. In \Cref{Sec2.6},  when estimating the maximal displacement of branching random walk, \Cref{A5} plays an essential role. Under \Cref{A5}, the time-inhomogeneous branching processes evolve very close to the supercritical cases, thus we can use the deviation results together with exponent increasing property to estimate the lower bound of $M_n^n$.
	\end{remark}

	\subsection{The Sample Path Moderate Deviation Principle}\label{Sec2.2}
	
	For any $\nu\in N_n^{n}$, we define following path re-scale.
	
	\begin{definition}\label{Path scale} 
		For any $\nu\in N_n^n$, for each $0\leqslant k\leqslant n$,
		use $V^n_\nu(k)$ denote the position of the $k$-generation ancestor of $\nu$, thus the path of $\nu$ is $ V^n_\nu:= \{ V^n_\nu(0),V^n_\nu(1),\cdots,V^n_\nu(n) \} $.
		Define the function $V_\nu^{n,a_n}$ on $[0,1]$
		to be the  re-scaled path of $\nu$,
		\begin{align*}
			V_\nu^{n,a_n}(t):=
			\begin{cases}
				\dfrac{1}{\sqrt{na_n}}
				\left(  \left(\abs{nt}+1-nt\right)V^n_\nu\left(\abs{nt}\right)+   \left(nt-\abs{nt}\right) V^n_\nu\left(\abs{nt}+1\right) \right)   , &t\in [0,1).
				\\
				\dfrac{1}{\sqrt{na_n}} V^n_\nu\left( n \right), & t=1.
			\end{cases} 
		\end{align*}
	\end{definition}
	
	For any $r\in [0,1]$, the space $C_0([0,r])$ is the set of continue functions on $[0,r]$ with value $0$ at $0$, and endow it with the supremum norm $\|\cdot\|$.
	For any $r>0$, set $H_r:=\left\{ g(t); g^\prime\in L_2[0,r] \right\}$ denote the space of all absolutely continuous function with value $0$ at $0$ that posses a square integrable derivation. Use $I_r(\cdot)$ denote the Schilder rate function of Brownian motion on $C_0([0,r])$:
	\begin{align*}
		I_r(g):= 
		\begin{cases}
			\frac{1}{2} \int_{0}^{r} g^{\prime}(s)^2 \d s \quad &g \in H_r
			\\
			+\infty &\text{otherwise.}
		\end{cases}
	\end{align*}
	According to the Schilder theorem,
	$I_r$ is a good rate function:
	for any $0\leqslant\alpha<\infty$, the level set $\{f\in C_0([0,r]): I_r(f)\leqslant \alpha \}$ is a compact subset of $C_0\left([0,r]\right)$.
	
	\
	
	For the path of time-inhomogeneous branching random walk,
	we have following results,
	which can be viewed as a kind of moderate deviation principle.
	
	\
	
	\begin{theorem}\label{Moderate Deviation}
		Under \Cref{A2}, 
		there is a moderate deviation principle for the path of the time-inhomogeneous branching random walk, i.e., 
		\begin{itemize}
			\item Upper bound:
			If $C$ is a closed subset of $C_0([0,1])$, then
			\[
			\limsup\limits_{n\rightarrow\infty} \dfrac{1}{a_n} \log 
			P_n \left( \exists \nu\in N_n^{n} : V_{\nu}^{n,a_n} \in C  \right)
			\leqslant -\inf\limits_{g\in C} S_f(g). 
			\]
			
			\item Lower bound: If $G$ is an open subset of $C_0([0,1])$, then
			\[
			\liminf\limits_{n\rightarrow\infty} \dfrac{1}{a_n} \log 
			P_n \left( \exists \nu\in N_n^{n} : V_{\nu}^{n,a_n}  \in G  \right)
			\geqslant -\inf\limits_{g\in G} S_f(g). 
			\]
		\end{itemize}
		Here $S_f(g) := \begin{cases}
			\sup\limits_{r \in [0,1]} \left\{ \int_{0}^{r} 
			\dfrac{1}{2}  g^{\prime}(s)^2 \d s- f(r) \right\}, \quad &\text{for } g\in H_1;
			\\
			+\infty, \quad  &\text{otherwise.}
		\end{cases}$
	\end{theorem}
	
	\
	
	\begin{remark}
		Hardy and Harris prove related results for binary branching Brownian motion in \cite{Harris2006}, and we just extend such path deviation principle to the general time-inhomogeneous branching system.
	\end{remark}
	
	\
	
	About the rate function $S_f(\cdot)$ in \Cref{Moderate Deviation}, we have 
	
	\
	
	\begin{theorem}\label{Good rate function}
		For any $\alpha\in[0,\infty)$, 
		the level set $ \{g:S_f(g)\leqslant \alpha\} $ 
		is compact in $C_0\left([0,1]\right)$, which means
		$S_f(g)$ 
		is a good rate function on $C_0\left([0,1]\right)$.
	\end{theorem}
	
	\begin{proof}[Proof of \Cref{Good rate function}]
		Fixed $\alpha\in[0,\infty)$,
		for any $g_n\in C_0\left([0,1]\right) $ 
		such that $S_f(g_n)\leqslant\alpha$, 
		then $I_1(g_n)\leqslant \alpha +f(1)$.
		And due to $ I_1(\cdot) $ is a good rate function, 
		we know 
		$ \{g:I_1(g)\leqslant\alpha+f(1)\} $ is a compact set.
		So, for any subsequence $g_{n_k}$ such that 
		$ g_{n_k} $ convergence to some function $ g $,
		we know $g$ also satisfies $I_1(g)\leqslant \alpha+f(1)<\infty$.
		
		For each $r\in[0,1]$, and $I_r$ is also a good rate function, confined to the interval $[0,r]$, we have
		\begin{align*}
			\int_{0}^{r} \dfrac{1}{2}g^{\prime}(s)^2 \d s - f(r)
			&=I_r(g) -f(r)
			\leqslant
			\liminf\limits_{k\rightarrow\infty}I_r(g_{n_k})-f(r)
			\nonumber\\
			&=
			\liminf\limits_{k\rightarrow\infty}
			\left[\int_{0}^{r}\dfrac{1}{2}g_{n_k}^{\prime}(s)^2 \d s - f(r)\right]
			\leqslant
			\liminf\limits_{k\rightarrow\infty}S_f(g_{n_k})
			\leqslant \alpha.
		\end{align*}
		Due to the arbitrariness of $r\in [0,1]$, then $S_f(g)\leqslant\alpha$.
		Due to the arbitrariness of subsequence,
		we know $\left\{ g:S_f(g)\leqslant \alpha \right\}$ is a compact set.
	\end{proof}
	
	\begin{remark}
		In Theorem 8.1 of \cite{Harris2006}, Hardy and Harris prove the deviation principle firstly, and then use the result to prove the goodness of the rate function. In \Cref{Good rate function}, relying on the goodness of Schilder rate function $I(\cdot)$, together with the definition of $S_f(\cdot)$, we can prove its goodness directly.
	\end{remark}
	
	\
	
	The proof of \Cref{Moderate Deviation} is mainly divided into three parts: in \Cref{Sec2.3}, depending on the many-to-one formula, we will prove the local upper bound; in \Cref{Sec2.4}, using the spine decomposition, we will prove the local lower bound; in \Cref{Sec2.5}, follow the steps of Section 7 in \cite{Harris2006}, firstly establish the weakly deviation principle and due to the goodness of rate function which means the exponential tightness, then \Cref{Moderate Deviation} will be proved.

	\subsection{Local Upper Bound}\label{Sec2.3}
	
	We at first need the path moderate deviation principle of the random walk.

	\begin{lemma}\label{MDP for RW}
		Assume $a_n\rightarrow\infty$ and $\frac{a_n}{n} \rightarrow 0$, let $(B_0:=0,B_1,\cdots,B_n)$ be a random walk with one step distribution $\mathcal{N}(0,1)$,
		denote $\mu_n$ as the law of 
		\[	B_n^{a_n}(\cdot):= \begin{cases}
			\dfrac{1}{\sqrt{na_n}} \left(  \left(\abs{nt}+1-nt\right) B_{\lfloor nt \rfloor} + \left(nt-\abs{nt}\right)B_{\lfloor nt \rfloor+1}  \right) ,&t\in[0,1);
			\\
			\dfrac{1}{\sqrt{na_n}} B_n, &t=1. \end{cases} \]
	    in $C_0([0,1])$, then
		$\{\mu_n\}$ satisfies the Large Deviation Principle with the good rate function $I_1(\cdot)$ at speed $a_n$ in $C_0([0,1])$.
	\end{lemma}
	
	\begin{proof}[Proof of \Cref{MDP for RW}]
		We can enlarge the probability space, let $\{B_t,t\in [0,\infty) \}$ is a standard Brownian motion.
		Note that $\mu_n$ is the law of $B^{a_n}_n(\cdot) $ in $C_0([0,1])$.
		Use $\nu_n$ denote the law of $ B^n(t) := \dfrac{1}{\sqrt{na_n}} B_{nt} $ in $C_0([0,1])$,
		due to the scale property of Brownian motion, $\nu_n $ is also the law of $ \dfrac{B_t}{\sqrt{a_n}} $ in $C_0([0,1])$.
		
		According to the Schilder Theorem (see Theorem 5.2.3 in \cite{LDP1998}), we know 
		$\nu_n$ satisfied the LDP with rate function $I_1(\cdot)$ at speed $a_n$ in $C_0([0,1])$.
		
		And for each $\epsilon>0$,
		\begin{align*}
			&\P( \exists t\in [0,1], \left|\dfrac{B_{nt}}{\sqrt{na_n}}- B^{a_n}_n(t)\right|>\epsilon )
			\\
			\leqslant
			&\sum_{k=1}^n\P\left( \exists t\in [k-1,k],|(B_t-B_k)-t(B_{k+1}-B_k)|>\epsilon \sqrt{na_n} \right)
			\\
			=&
			n\P^{0,0}_{1,0}\left( \max\limits_{t\in [0,1]} |B_t|> \epsilon \sqrt{na_n} \right)
			\\
			\leqslant
			&n\P^{0,0}_{1,0}\left( \max\limits_{t\in [0,1]} B_t> \epsilon \sqrt{na_n} \right)
			+n
			\P^{0,0}_{1,0}\left( \min\limits_{t\in [0,1]} B_t< -\epsilon \sqrt{na_n} \right)
			\\
			=
			&2n\P^{0,0}_{1,0}\left( \max\limits_{t\in [0,1]} B_t> \epsilon \sqrt{na_n} \right) = 2n \exp\left( -2\epsilon^2na_n \right)
		\end{align*}
		Here under $ \P^{0,0}_{1,0} $, $\{B_t,t\in [0,1]\}$ is a Brownian bride form $0$ at time $0$ to $0$ at time $1$, and the last equality is based on the the probability of a Brownian bridge to stay below a straight line.
		And it follows, by considering $n\rightarrow\infty$, that
		\begin{equation*}
			\lim\limits_{n\rightarrow\infty}
			\frac{1}{a_n}\log \P( \exists t\in [0,1], \left|\dfrac{B_{nt}}{\sqrt{na_n}}- B_n^{a_n}(t)\right|>\epsilon )
			= -\infty.
		\end{equation*}
		Therefore, the probability measure $\mu_n$ and $\nu_n$ are exponentially equivalent. 
		By Theorem 4.2.13 in \cite{LDP1998}, 
		the proof of \Cref{MDP for RW} now is over.
	\end{proof}
	
	\begin{remark}
		In fact, for normal distribution $\ND(0,1)$, the condition $\lim\limits_{n\rightarrow\infty} \frac{a_n}{n} = 0$ can be omitted in \Cref{MDP for RW}.
		And for more general random walk which satisfies the Cram$\acute{e}$r's condition, together with the moderate deviation (see Theorem 3.7.1 in \cite{LDP1998}, and the condition $\lim\limits_{n\rightarrow\infty} \frac{a_n}{n} = 0$ is important), we can prove \Cref{MDP for RW} by repeating the proof of Theorem 5.1.2 in \cite{LDP1998}.
	\end{remark}
	
	\
	
	Now we can prove  the local upper bound of \Cref{Moderate Deviation}.
	
	\begin{theorem}\label{local upper bound}
		For any $g\in C_0([0,1])$, we have
		\begin{equation}\label{eq upper bound}
			\lim\limits_{\delta\rightarrow 0}
			\limsup\limits_{n\rightarrow\infty} 
			\dfrac{1}{a_n} \log P_n\left( \exists \nu\in N_n^{n} : V_{\nu}^{n,a_n} \in  \mathbb{B}_\delta(g)   \right)
			\leqslant - S_f(g).
		\end{equation}
		Here $ \mathbb{B}_\delta(g) := 
		\{\rho\in C_0\left([0,1]\right): 
		\left| \rho(t)-g(t)\right|<\delta, \forall t\in [0,1]\}$.
	\end{theorem}

	\begin{proof}[Proof of \Cref{local upper bound}]
		First note that as $\delta\rightarrow0$, \cref{eq upper bound} is decreasing, which means the limit does exist.
		
		Under $P_n$, the sequence $ (B_0,B_1,\cdots,B_n)$ 
		is a random walk with one step distribution $\mathcal{N}(0,1)$.
		For each $r\in[0,1]$,
		Using many-to-one formula (in fact, calculate the mean of branching random walk), 
		\begin{align}
			P_n\left( \exists \nu\in N_n^{n} : V_{\nu}^{n,a_n}\right.&\left. \in  \mathbb{B}_\delta(g) \right)
			\leqslant
			P_n\left(  \exists \nu\in N_n^{n} : |V_{\nu}^{n,a_n}(s) - g(s)| <\delta ,\forall s \in [0,r] \right)
			\nonumber
			\\
			&\leqslant
			P_n\left(  \exists \nu\in N_{\abs{nr}}^{n} : 
			\left|V_{\nu}^{n,a_n}(s) - g(s)\right| <\delta,
			\forall s \in [0,r] \right)
			\nonumber
			\\
			&\leqslant
			E_n\left( \sum\limits_{\nu\in N_{\abs{nr}}^{n}} 
			\ind{|V_{\nu}^{n,a_n}(s) - g(s)| <\delta ,\forall s \in [0,r]} \right)
			\nonumber\\
			&=
			\exp\left(S^n_{\lfloor nr \rfloor}\right)
			P_n \left( \left| B^{a_n}_n(s) - g(s)\right| <\delta ,\forall s \in [0,r] \right)
			\label{omega estimate}
		\end{align}
		Here $B_n^{a_n}(\cdot)$ is defined in \Cref{MDP for RW}, and according to \Cref{MDP for RW}, we know
		\begin{equation*}
			\lim\limits_{\delta\rightarrow0}
			\limsup\limits_{n\rightarrow\infty}\dfrac{1}{a_n}
			\log P_n \left( \left|B^{a_n}_n(s) - g(s)\right| <\delta ,\forall s \in [0,r] \right)
			\leqslant -\int_{0}^{r} \dfrac{1}{2}g^{\prime}(s)^2 \d s.
		\end{equation*}
		
		Together with \cref{omega estimate} and \Cref{A2}, then
		\begin{align*}
			\lim\limits_{\delta\rightarrow 0}
			\limsup\limits_{n\rightarrow\infty} 
			\dfrac{1}{a_n} \log P_n 
			\left( \exists \nu\in N_n^{n} : 
			V_{\nu}^{n,a_n} \in  \mathbb{B}_\delta(g) \right)
			\leqslant  -\left(\int_{0}^{r} \dfrac{1}{2}g^{\prime}(s)^2 \d s- f(r)\right).
		\end{align*}
		
		Due to the arbitrary of $r\in [0,1]$, 
		the proof of \Cref{local upper bound} is over.
	\end{proof}

	\subsection{Local lower bound}\label{Sec2.4}
	
	The following theorem describe the lower bound of \Cref{Moderate Deviation}.
	
	\
	
	\begin{theorem}\label{local lower bound}
		For $g\in C_0([0,1])$,
		\begin{equation*}
			\liminf\limits_{n\rightarrow\infty} 
			\dfrac{1}{a_n} \log P_n\left( \exists \nu\in N_n^{n} : 
			V_{\nu}^{n,a_n} \in  \mathbb{B}_\delta(g) \right)
			\geqslant -S_f(g).
		\end{equation*}
	\end{theorem}
	
	In order to prove \Cref{local lower bound}, we will change the measure to obtain a realization of target event. And in branching system, it is a very mature approach to do such by additive martingale.
	
	Choose a function $g\in H_1$,
	then $\int_{0}^{1}\frac{1}{2}g^{\prime}(s)^2 \d s <\infty $.
	Consider the scale-up of $g(t)$, for each $0\leqslant k\leqslant n$, let
	\[
	g^{n,a_n}(k) :=  \sqrt{na_n} g\left(\frac{k}{n}\right),
	\quad \text{and}
	\quad \lambda^n_i = g^{n,a_n}(i)-g^{n,a_n}(i-1), \quad 1\leqslant i\leqslant n.
	\]
	and define $W_0^n:=0$, and for $1\leqslant k\leqslant n$,
	\begin{align*}
		W_k^{n}:= \sum\limits_{\nu\in N_k^{n}}
		\exp(-S^n_k) \exp\left( \sum_{i=1}^{k} \lambda^n_i \left[X^n(\nu(i))-X^n(\nu(i-1))\right]-\frac{1}{2}(\lambda_i^n)^2 \right)
	\end{align*}
	Under $P_n$, the sequence $\left\{W_0^{n},W_1^{n},\cdots,W_n^{n}\right\}$ construct the so-called additive martingale. 
	Thus, we can construct a new measure $Q_n$ which satisfies 
	$\dfrac{\d Q_n}{\d P_n}\big|_{\mathcal{F}^n_k}:= W_k^{n}$, 
	here the filtration $\{\mathcal{F}^n_k:0\leqslant k\leqslant n\}$ 
	is the natural filtration of the $n$-th branching random walk. What's more, under $Q_n$, the branching system is a branching process with the spine and can be described as follow:
	\begin{itemize}
		\item Initially, single particle stays at $0$, 
		and this particle is the 0-generation spine particle,
		denote it as $\omega_0^n$.
		\item At generation $k$, the spine particle $\omega^n_{k-1}$ dies, 
		and gives birth to children according to 
		the law $\{q^n_k[s]:=\dfrac{sF^n_k[s]}{\overline{F}^n_k};s=1,2,\cdots\}$ 
		(the size-biased distribution of $F_k^n$),
		then uniformly choose one as the $k$-generation spine particle, 
		denote it as $\omega^n_k$, other particles are normal.
		Besides the spine particle $\omega^n_{k-1}$, 
		other normal particles die and reproduce just 
		according to $F^n_k$ at generation $k$, 
		and all these reproduced particles are normal.
		\item The spine particle $\omega_k^n$ move 
		according to $\ND(\lambda_k^n,1)$ respect to its parent $\omega^n_{k-1}$; all normal particles move according to $\ND(0,1)$ respect to their own parent.
	\end{itemize}
	
	\begin{remark}\label{remark 1}
		Under $Q_n$,
		define $\hat{V}^n(\omega_k^n):=V^n(\omega_k^n)-g^{n,a_n}(k)$,
		then $\{\hat{V}^n(\omega_0^n),\hat{V}^n(\omega_1^n),
		\cdots,\hat{V}^n(\omega_n^n)\} $ 
		is just a random walk with one step distribution $\ND(0,1)$.
	\end{remark}
	
	About $P_n$ and $Q_n$, we have
	\begin{align}\label{Des Prob using Q_n}
		P_n\left(  \exists \nu\in N_n^{n} : 
		V_{\nu}^{n,a_n} \in  \mathbb{B}_\delta(g)  \right)
		=Q_n\left( \dfrac{1}{W_n^{n}}; 
		\exists \nu\in N_n^{n} : V_{\nu}^{n,a_n} \in  \mathbb{B}_\delta(g)  
		\right)
	\end{align}
	
	Now we calculate  the $\alpha$-moment of $W_n^n$ under $Q_n$.
	\begin{lemma}\label{M of M_n under Q}
		Under \Cref{A2}, 
		for any $\alpha\in [0,1]$, and for any $\delta>0$, 
		there exist $N(\delta)$ such that for all $n>N(\delta)$,
		\begin{align}\label{MF of Mn under Q}
			Q_n\left( (W_n^{n})^\alpha \right)
			\leqslant
			\exp(a_n\delta)\exp\left(\alpha a_nS_f(g)\right)
			\exp\left( \frac{1}{2}\alpha^2a_n\int_{0}^1g^{\prime}(s)^2 \d s \right)
			\left(1+\sum_{k=1}^n\overline{F}^n_k\tilde{F}^n_k \right)
		\end{align}
	\end{lemma}
	
	\begin{proof}[Proof of \Cref{M of M_n under Q}]
		Let the filtration $\mathcal{G}_\infty$ contains
		all information about the spine, and therefore we can obtain
		the spine decomposition:
		\begin{align}
			&\Q\left( W_n^{n} |\mathcal{G}_\infty\right)
			=
			\exp\left( -S^n_n \right)
			\exp\left(\sum_{i=1}^n\lambda^n_i
			\left(V^n(\omega^n_i)-V^n(\omega^n_{i-1})\right)-\frac{1}{2}(\lambda_i^n)^2\right)
			\nonumber\\
			+&
			\sum_{k=1}^{n} \exp\left(-S_{k}^n\right)
			\left(N(\omega_{k-1}^n)-1\right)
			\exp\left(\sum_{i=1}^{k-1}\lambda_i^n
			\left(V^n(\omega^n_i)-V^n(\omega^n_{i-1})\right)-\frac{1}{2}(\lambda_i^n)^2\right)
			\label{eq10}
		\end{align}
		Here $ N(\omega_{k-1}^n) $ represent the number of children 
		that $\omega_{k-1}^n$ has.
		Recall the following well-known facts.
		
		\begin{proposition}\label{Prop}
			If $\alpha\in [0,1]$ and $u,v>0$
			then $(u+v)^\alpha \leqslant u^\alpha+v^\alpha$.
		\end{proposition}
		
		Note that $Q_n\left[ N(\omega_{k-1}^n)-1 \right] = \overline{F}^n_k\tilde{F}^n_k$, by Jeason's inequality, \cref{eq10}, \Cref{remark 1} and \Cref{Prop}, we have
		\begin{align}
			&~\quad Q_n\left(\left(W_n^{n}\right)^\alpha\right)
			=
			Q_n\left( Q_n\left( (W_n^{n})^\alpha\left|\right.
			\mathcal{G}_\infty \right) \right)
			\leqslant
			Q_n\left( Q_n\left( W_n^{n}|\mathcal{G}_\infty \right)^\alpha \right)
			\nonumber\\
			&\leqslant
			Q_n\left(  \exp\left( -\alpha S^n_n \right)
			\exp\left(
			\sum_{i=1}^n\alpha\lambda^n_i(\hat{V}^n(\omega_i^n)-\hat{V}^n(\omega^n_{i-1}))
			+\frac{1}{2}\alpha(\lambda_i^n)^2\right)  \right)
			\nonumber\\
			&+
			\sum_{k=1}^{n}
			Q_n\left( \exp(-\alpha S_{k}^n)
			\left(N(\omega_{k-1}^n)-1\right)
			\exp
			\left(\sum_{i=1}^{k-1}\alpha\lambda^n_i(\hat{V}^n(\xi_i)-\hat{V}^n(\xi_{i-1}))
			+\frac{1}{2}\alpha(\lambda_i^n)^2\right) \right)
			\nonumber\\
			&=
			\exp\left(-\alpha S_n^n\right)
			\exp\left(\dfrac{\alpha^2+\alpha}{2}\sum_{i=1}^n (\lambda_i^n)^2\right)
			+\sum_{k=1}^{n}\exp(-\alpha S^n_k)\overline{F}^n_k\tilde{F}^n_k
			\exp\left( \dfrac{\alpha^2+\alpha}{2}\sum_{i=1}^{k-1} (\lambda_i^n)^2 \right)
			\label{eq11}
		\end{align}
		
		According to \Cref{Remark of Convergence},
		for any $\delta>0$, there exist $N$
		such that for any $n>N$,$ \left|\dfrac{1}{a_n}S^n_k-f(\dfrac{k}{n})\right| < \dfrac{\delta}{\alpha} $, then
		\begin{align}\label{eq12}
			\alpha S^n_k\geqslant \alpha a_nf(\frac{k}{n})-a_n\delta 
		\end{align}
		What's more, for any $0\leqslant k \leqslant n$, 
		use Cauchy-Schwarz Inequality,
		\begin{align}\label{eq13}
			\sum_{i=1}^k (\lambda_i^n)^2 =
			na_n\sum_{i=1}^k
			\left( g(\dfrac{i}{n})-g(\dfrac{i-1}{n}) \right)^2
			\leqslant a_n\int_{0}^{\frac{k}{n}} g^{\prime}(s)^2 \d s.
		\end{align}
		Then according to \cref{eq11,eq12,eq13} and the definition of $S_f(g)$,
		\begin{align}
			Q_n\left((W_n^{n})^\alpha\right)
			&\leqslant
			\exp\left(a_n\delta-\alpha a_n f(1)\right)
			\exp\left(\dfrac{\alpha^2+\alpha}{2} a_n \int_{0}^{1} g^{\prime}(s)^2 \d s \right)
			\nonumber\\
			&+
			\sum_{k=1}^{n}
			\exp\left(a_n\delta-\alpha a_n f(\dfrac{k-1}{n})\right)
			\overline{F}^n_k\tilde{F}^n_k
			\exp
			\left( \dfrac{\alpha^2+\alpha}{2}
			a_n\int_{0}^{\frac{k-1}{n}} g^{\prime}(s)^2 \d s \right)
			\nonumber\\
			&\leqslant
			\exp(a_n\delta)\exp\left(\alpha a_n S_f(g)\right)
			\exp\left(\alpha^2 a_n\int_{0}^{1} \frac{1}{2}g^{\prime}(s)^2\d s\right)
			\left(1+\sum_{k=1}^{n}\overline{F}^n_k\tilde{F}^n_k\right)\nonumber
		\end{align}
		Thus, the proof is over.
	\end{proof}

	Using \Cref{M of M_n under Q}, we have
	\begin{theorem}\label{Theorem M_n under Q}
		Under \Cref{A2},
		for any $\epsilon>0$,
		\begin{align*}
			\lim\limits_{n\rightarrow\infty}
			Q_n\left(\sup\limits_{0\leqslant k\leqslant n} W_k^n 
			\leqslant \exp\left(a_n \left(S_f(g)+\epsilon\right)\right)\right) = 1.
		\end{align*}
	\end{theorem}
	
	\begin{proof}[Proof of \Cref{Theorem M_n under Q}]
		For each $\alpha\in (0,1)$, $(W_k^n)^{\alpha}$ is a $Q_n$ submartingale (indeed, $(W_k^n)^{1+\alpha}$ is $P_n$ submartingale), then according to Doob's submartingale inequality,
		\begin{align}
			Q_n\left(\max\limits_{0\leqslant k\leqslant n} W_k^{n} > \exp\left(a_n(S_f(g)+\epsilon)\right) \right)
			&=
			Q_n\left( \max\limits_{0\leqslant k\leqslant n} (W_k^{n})^\alpha >
			\exp\left( a_n\alpha(S_f(g)+\epsilon)\right) \right)
			\nonumber\\
			&\leqslant
			\dfrac{Q_n\left( (W_n^{n})^\alpha \right)}
			{\exp\left(a_n\alpha(S_f(g)+\epsilon)\right)}\label{eq14}
		\end{align}
		
		According to \Cref{M of M_n under Q}, together with \cref{eq14},
		\begin{align}
			Q_n&\left(\max\limits_{0\leqslant k\leqslant n} W_k^{n} > \exp\left(a_n(S_f(g)+\epsilon)\right) \right)
			\nonumber\\
			&\leqslant
			\exp(a_n\delta-a_n\epsilon)\exp\left( \frac{1}{2}\alpha^2a_n\int_{0}^1g^{\prime}(s)^2 \d s \right) \left(1+\sum_{k=1}^n\overline{F}^n_k\tilde{F}^n_k\right)\label{eq15}
		\end{align}
		
		According to \Cref{A2}, due to the continuity of $f$, we know $\log \left(\max\limits_{1\leqslant k\leqslant n}\{\overline{F}^n_k\}\right) = o(a_n)$, thus
		\begin{align*}
			\log\left(1+\sum_{k=1}^n\overline{F}^n_k\tilde{F}^n_k\right)
			\leqslant
			\log\left(1+\sum_{k=1}^n\tilde{F}^n_k\right)+\log \left(\max\limits_{1\leqslant k\leqslant n}\{\overline{F}^n_k\}\right) = o(a_n).
		\end{align*}
		Now choose $\alpha$ small enough, 
		such that $\alpha \int_{0}^1 \frac{1}{2}g^{\prime}(s)^2\d s -\epsilon <0 $, then fix some $ \delta $ small enough, which satisfies $\delta < \alpha \left( \epsilon- \alpha\int_{0}^1\frac{1}{2}g^{\prime}(s)\d s \right)$. In \cref{eq15}, take $n\rightarrow\infty$, then
		\begin{equation*}
			\lim\limits_{n\rightarrow\infty}
			Q_n\left(\max\limits_{0\leqslant k\leqslant n} W_k^{n} 
			> \exp\left(a_n(S_f(g)+\epsilon)\right) \right)
			= 0.
		\end{equation*}
		The proof is finished.
	\end{proof}
	
	\
	
	Now, we are going to prove \Cref{local lower bound}:
	\begin{proof}[Proof of \Cref{local lower bound}]
		According to \cref{Des Prob using Q_n},
		\begin{align}
			&P_n\left(\exists \nu\in N_n^{n} : V_{\nu}^{n,a_n} \in  \mathbb{B}_\delta(g)  \right)
			=Q_n\left( \dfrac{1}{W_n^{n}}; \exists \nu\in N_n^{n} : V_{\nu}^{n,a_n} \in  \mathbb{B}_\delta(g)  \right)
			\nonumber\\
			\geqslant&
			\exp(-a_n(S_f(g)+\epsilon)) Q_n\left( \max\limits_{0\leqslant k\leqslant n} W_k^{n} \leqslant\exp\left(a_n\left(S_f(g)+\epsilon\right)\right);\right.
			\nonumber\\
			&\quad \bigg. \quad\quad\quad\quad\quad\quad\quad\quad\quad\quad\quad\quad\quad
			\exists \nu\in N_n^{n}, ~ V_{\nu}^{n,a_n} \in  \mathbb{B}_\delta(g) \bigg)
			\nonumber\\
			\geqslant&
			\exp(-a_n(S_f(g)+\epsilon)) Q_n\left( \max\limits_{0\leqslant k\leqslant n} W_k^{n} \leqslant\exp\left(a_n\left(S_f(g)+\epsilon\right)\right);\right.
			\nonumber\\
			&\quad \quad\quad\quad\quad\quad\quad\quad\quad\quad\quad\quad\quad\quad \left. \omega_n^n\in N_n^{n}, ~ V_{\omega_n^n}^{n,a_n} \in \mathbb{B}_\delta(g) \right).\label{eq16}
		\end{align}
		
		According to \Cref{remark 1} and the law of large number, we know
		\begin{equation*}
			\lim\limits_{n\rightarrow\infty}
			Q_n\left( \omega_n^n\in N_n^{n}, ~ V_{\omega_n^n}^{n,a_n} \in \mathbb{B}_\delta(g) \right) = 1.
		\end{equation*}
		Together with \Cref{Theorem M_n under Q} and \cref{eq16},
		\begin{equation*}
			\liminf\limits_{n\rightarrow\infty} 
			\dfrac{1}{a_n} \log P_n\left( \exists \nu\in N_n^{n} : 
			V_{\nu}^{n,a_n} \in  \mathbb{B}_\delta(g)   \right)
			\geqslant - S_f(g)-\epsilon.
		\end{equation*}
		
		Due to the arbitrary of $\epsilon$, 
		the proof of \Cref{local lower bound} is over.
	\end{proof}

	\subsection{Improving the ``weak'' deviation result}\label{Sec2.5}
	
	In this subsection, we will prove \Cref{Moderate Deviation}. As the same structure in \cite{Harris2006}, repeat the Section 7 in \cite{Harris2006}, we can prove the weak deviation results, i.e. the upper bound in \Cref{Moderate Deviation} holds for all compact subset.
	\begin{theorem}\label{weak DP}
		The results in \Cref{local upper bound} and \Cref{local lower bound} means that the upper bound of \Cref{Moderate Deviation} hold for all $C\in C_0([0,1])$ that are closed and compact:
		\begin{equation*}
			\limsup\limits_{n\rightarrow\infty} \dfrac{1}{a_n} \log 
			P_n \left( \exists \nu\in N_n^{n} : V_{\nu}^{n,a_n} \in C  \right)
			\leqslant -\inf\limits_{g\in C} S_f(g). 
		\end{equation*}
		whilst the lower bound holds in full for all open subsets $G\in C_0([0,1])$:
		\begin{equation*}
			\liminf\limits_{n\rightarrow\infty} \dfrac{1}{a_n} \log 
			P_n \left( \exists \nu\in N_n^{n} : V_{\nu}^{n,a_n}  \in G  \right)
			\geqslant -\inf\limits_{g\in G} S_f(g). 
		\end{equation*}
	\end{theorem}
	
	\begin{proof}[Proof of \Cref{weak DP}]
		There is no more different that need to be dealt with specifically.
		Following the Section 7 in \cite{Harris2006}, we can prove it directly.
	\end{proof}
	
	Now we are going to prove \Cref{Moderate Deviation}.
	
	\begin{proof}[Proof of \Cref{Moderate Deviation}]
		Due to the many-to-one formula, for any compact subset $K \subset C_0([0,1])$, we have
		\begin{align*}
			P_n\left( \exists \nu \in N_n^n: V_\nu^{n,a_n}\in K \right)
			&\leqslant
			E_n\left( \# \left\{ \nu \in N_n^n: V_\nu^{n,a_n}\in K\right\} \right)
			\\
			&=\exp(S_n^n) P_n\left( B^{a_n}_n(\cdot) \in K \right)
		\end{align*}
		Here $B^{a_n}_n(\cdot)$ is defined in \Cref{MDP for RW}.
		Then
		\begin{align*}
			\limsup\limits_{n\rightarrow\infty}
			\dfrac{1}{a_n}\log P_n\left( \exists \nu \in N_n^n: V_\nu^{n,a_n}\in K \right)
			\leqslant
			f(1)
			+ \limsup\limits_{n\rightarrow\infty}
			\dfrac{1}{a_n}\log
			P_n\left( B^{a_n}_n(\cdot) \in K \right)
		\end{align*}
		Then, due to \Cref{MDP for RW}, \Cref{A2} ($f(1)<\infty$) and the goodness of $I_1(\cdot)$, we can obtain the the exponential tight property about the (sub-additive) measure $\left\{P_n\left( \exists \nu \in N_n^n: V_\nu^{n,a_n}\in \cdot \right) \right\}_{n\geqslant1}$.
		
		Using the \Cref{weak DP}, together with the exponential tight property, following the proof in Lemma 1.2.18 in Chapter 1 of  \cite{LDP1998}, we can prove \Cref{Moderate Deviation}.
	\end{proof}

	\subsection{The maximal displacement}\label{Sec2.6}
	
	In this subsection, we want to estimate the maximal displacement $M_n^n$  of the time-inhomogeneous branching random walk under $P_n$, and prove a limit theorem.
	
	For the continuous function $f$ in \Cref{A2},
	define
	\begin{equation}\label{representation A_f}
		A_f:= \sup \left\{ g(1):\int_{0}^{r} 
		\dfrac{1}{2} g^{\prime}(s)^2 \d s\leqslant f(r), 
		\forall r \in [0,1] \right\}
	\end{equation}
	For the $n$-th branching random walk, 
	the maximal displacement in time $n$, 
	i.e. $M_n^{n}:=\max \{ V_\nu; \nu\in N_n^{n}\} $:
	\begin{theorem}\label{Max Dis of TIBRW}
		Under \Cref{A2} and \Cref{A5}, 
		the maximal displacement of 
		the time-inhomogeneous branching random walk satisfies
		\begin{equation*}
			\dfrac{M_n^{n}}{\sqrt{na_n}} \Longrightarrow A_f, \quad \text{as} \quad n\rightarrow\infty.
		\end{equation*}
	\end{theorem}
	
	\begin{remark}
		For normal distribution $\ND(0,1)$, the condition $\lim\limits_{n\rightarrow\infty} \frac{a_n}{n} = 0$ about $\{a_n\}$ in \Cref{A2} can be omitted (see \Cref{Remark an}).
		In \cite{Hammersley}, Hammersley proved the almost surely limit for the maximal displacement at time $n$ of supercritical branching random walk by using the sub-additive ergodic theorem.
		If the jump distribution is $\ND(0,1)$, and take $ a_n = n$, then \Cref{Max Dis of TIBRW} is a kind of weak law of large number for supercritical branching random walk. In detail, assume the mean of offspring distribution is $m>1$,
		it is easy to calculate that the log-Laplacian transform of the offspring point process is $\Lambda(\lambda)= \log m+ \frac{1}{2}\lambda^2$, accounting to the law of large number in \cite{Hammersley},
		then $\dfrac{M_n}{n} \stackrel{a.s}{\longrightarrow} \inf\left\{ \frac{\Lambda(\lambda)}{\lambda}\right\} = \sqrt{2\log m}$.
		In the setting of \Cref{Max Dis of TIBRW}, take $a_n=n$ and $f(t)=t\log m$, it is not hard to verify the supercritical branching random walk does satisfy \Cref{A2} and \Cref{A5}, then solve \cref{representation A_f}, we know $A_f= \sqrt{2\log m}$, then according to \Cref{Max Dis of TIBRW}, $\dfrac{M_n}{n} \dcon \sqrt{2\log m}$, which also means the convergence in probability.
	\end{remark}
	
	\
	
	\begin{proof}[Proof of \Cref{Max Dis of TIBRW}]
		For any $\epsilon>0$,
		let $C:=\left\{g\in C_0([0,1]); g(1)\geqslant A_f+\epsilon\right\} $,
		and $C$ is a closed subset of $C_0([0,1])$.
		
		We claim that  $ \inf\limits_{g\in C}S_f(g)>0$. Otherwise, if  $\inf\limits_{g\in C}S_f(g) = 0$, 
		there exist a sequence function $\{g_n\}\subset C$ such that 
		$\lim\limits_{n\rightarrow\infty}S_f(g_n)= 0$.
		Due to $S_f$ is a good rate function, then there is some sub-sequence function $\{g_{n_k}\}$ that convergence to some $g$ uniformly and $S_f(g)=0$.
		Because $C$ is a closed set, $g\in C$, then $ g(1)\geqslant A_f+\epsilon $,	which contradicts the definition of $A_f$. 
		According to the deviation principle, we have
		\begin{align*}
			\limsup\limits_{n\rightarrow\infty}\dfrac{1}{a_n} \log P_n\left( \dfrac{M_n^n}{\sqrt{na_n}}\geqslant A_f+\epsilon \right)
			&=
			\limsup\limits_{n\rightarrow\infty}\dfrac{1}{a_n} 
			\log P_n\left( \exists \nu\in N_n^n; V_\nu^{n,a_n}\in C\right)
			\nonumber\\
			&\leqslant -\inf\limits_{g\in C} S_f(g)<0.
		\end{align*}
		Then
		\begin{equation}\label{Upper of dis convergence}
			\lim\limits_{n\rightarrow\infty}
			P_n\left(\dfrac{M_n^{n}}{\sqrt{na_n}} > A_f +\epsilon \right) = 0.
		\end{equation}
		
		\
		
		Next, for any $\epsilon >0$,  we want to prove
		\begin{equation}\label{Lower of dis convergence}
			\lim\limits_{n\rightarrow\infty}
			P_n\left(\dfrac{M_n^{n}}{\sqrt{na_n}} \geqslant A_f-4\epsilon \right) = 1.
		\end{equation}
		
		According to \Cref{A2}, there exist some $h_0\in (0,1)$, such that  $f(h)<\frac{\epsilon^2}{4}$ for any $h\in(0,h_0)$, 
		then for any $h\in (0,h_0)$,
		\begin{align}
			P_n\left( \exists \nu\in N_n^{n}, 
			V_\nu^{n,a_n}(h) \leqslant -\epsilon \right)
			&\leqslant
			P_n \left( \exists \nu\in N_{\abs{nh}}^{n}, 
			\dfrac{V^n_\nu}{\sqrt{na_n}} \leqslant -\epsilon \right)
			\nonumber\\
			&\leqslant
			E_n \left( \#\left\{ \nu\in N_{\abs{nh}}^{n}, 
			\dfrac{V^n_\nu}{\sqrt{na_n}} \leqslant -\epsilon \right\} \right)
			\nonumber\\
			&=
			\exp\left(S^n_{\abs{nh}}\right) 
			P_n\left( B^{a_n}_n(h)
			\leqslant -\epsilon \right)
			\nonumber
		\end{align}
		Here $B^{a_n}_n(\cdot)$ is defined in \Cref{MDP for RW}.
		Under \Cref{A2}, due to the choice of $h$, which means for $n$ large enough, we have $S_{\abs{nh}}^n < \frac{a_n\epsilon^2}{3}$,
		and for normal distribution, we know $\P(\ND(0,\delta^2)<-x) < \exp(-\frac{x^2}{2\delta^2})$ for any $x>0$, as $n\rightarrow\infty$, we have
		\begin{align}
			\lim\limits_{n\rightarrow\infty}
			P_n\left(\forall \nu\in N_n^{n}, V_\nu^{n,a_n}(h) \leqslant-\epsilon \right)
			&\leqslant
			\lim\limits_{n\rightarrow\infty} 
			\exp\left(S^n_{\abs{nh}}\right) 
			P_n\left( B^{a_n}_n(h)
			\leqslant -\epsilon \right)
			\nonumber\\
			&\leqslant
			\lim\limits_{n\rightarrow\infty} 
			\exp\left(\frac{a_n\epsilon^2}{3}\right)
			\exp\left( -\frac{\epsilon^2a_n}{2h} \right)
			\nonumber\\
			&\leqslant
			\lim\limits_{n\rightarrow\infty} 
			\exp\left(\frac{a_n\epsilon^2}{3}\right)
			\exp\left( -\frac{a_n\epsilon^2}{2} \right)
			=0.\nonumber
		\end{align}
		So, for any $\epsilon>0$, there exist $h_0\in (0,1)$, such that for any $h\in (0,h_0)$, we have
		\begin{equation}
			\lim\limits_{n\rightarrow\infty}
			P_n\left(\forall \nu\in N_n^{n}, V_\nu^{n,a_n}(h)>-\epsilon \right)
			=1.\label{eq20}
		\end{equation}
		
		For each $n$,
		according to the definition of $A_f$, 
		there exist a function $g_n\in C_0[0,1]$ 
		such that $S_f(g_n) = 0$ and $ A_f\geqslant g_n(1)> A_f-\frac{1}{n}$.
		
		According to the goodness of $S_f$,
		there exist some sub-sequence $\{g_{n_k}\}$ 
		convergence to some function $g$ in $C_0\left([0,1]\right)$, 
		then $g(1)=\lim\limits_{k\rightarrow\infty}g_{n_k}(1)=A_f$,
		and $S_f(g)=0$.
		
		For the function $g$,
		define an increasing function $Q(y)$ as follow:
		\begin{equation*}
			Q(y):= \int_{0}^{y} \dfrac{1}{2}g^{\prime}(s)^2 \d s.
		\end{equation*}
		For $h$ small enough (in fact $h<Q(1)$),
		we can define $\phi(h)$ as the generalized inverse function of $Q(y)$ 
		as follow:
		\begin{equation*}
			\phi(h):=\inf\left\{y\in [0,1]; Q(y)>h\right\}.
		\end{equation*}
		About the function $Q(y)$ and $\phi(h)$, we have
		\begin{lemma}\label{Q and phi}
			For any $y>0$, $Q(y)>0$, 
			then $\phi(0)=0$.
		\end{lemma}
		
		The \Cref{Q and phi} will be proved after finishing the proof of \Cref{Max Dis of TIBRW}.
		
		For $h$ small enough such that $f(h)<Q(1)$,
		define $g_h(s)\in C_0\left([0,1]\right)$ with deviation as follow:
		\[
		g^\prime_h(t):=
		\begin{cases}
			0, \quad &t\in[0,\phi(f(h))];
			\\
			g^{\prime}(t), \quad &t\in [\phi(f(h)),1].
		\end{cases}
		\]
		Thus $g_h(1)= g(1)-g(\phi(f(h))) $.
		Due to $S_f(g)=0$,
		according to the definition of $\phi$, then $\phi(f(h))\geqslant h$, 
		and $f(h)=Q(\phi(f(h)))$.
		
		And for any $\gamma \in N_{\abs{nh}}^n $, for any $\epsilon>0$, define
		\begin{align*}
			\psi_n(h,g,\epsilon)(\gamma)
			:= P_n\left( \exists \nu \in N_n^n,\gamma\prec\nu,\right.&
			V_{\nu}^{n,a_n}(s)-V_{\nu}^{n,a_n}(h)\\
			&\left.> g_h(s)-\epsilon, 
			\forall s\in[h,1]\left|\right.\mathcal{F}_{\abs{nh}}\right)
		\end{align*}
		Here $\gamma\prec\nu $ means $\gamma$ is an ancestor of $\nu$.
		According to branching property, 
		we know above representation doesn't depend on the choose of $\gamma$, so just denote the value as $\psi_n(h,g,\epsilon)$.
		
		It is not hard to verified the deviation principle also works
		for the time-inhomogeneous branching random walk 
		that roots at any $\gamma\in N_{\abs{nh}}^n$ (just repeat the previous argument).
		So, according to the deviation principle,
		for any open set $V$ in $C_0([0,1-h];\R)$, we know
		\begin{align*}
			\liminf\limits_{n\rightarrow\infty}\dfrac{1}{a_n}
			\log &P_n\left(
			\exists \nu \in N_n^n,\gamma\prec\nu, V_\nu^{n,a_n}(h+\cdot)-V_\nu^{n,a_n}(h) \in V
			\right)
			\\
			&\geqslant
			-\inf\limits_{k\in V} \sup\left( \int_{0}^{\omega} \dfrac{1}{2} k^{\prime}(s)^2\d s-\left(f(h+\omega)-f(h)\right)\right)
		\end{align*}
		
		Choosing $V:=\left\{ k(\cdot)\in C_0\left([0,1-h];\R\right): 
		k(s) > g_h(h+s)-\epsilon, \forall s\in[0,1-h] \right\}$
		and $g_h(h+\cdot)\in V$.
		Note that
		\begin{equation*}
			\int_{0}^{r} \dfrac{1}{2} g_h^{\prime}(h+s)^2 \d s - f(h+r)+f(h)
			\leqslant
			\begin{cases}
				0-f(h+r)+f(h)\leqslant 0, \\
				\quad\quad \quad \quad  r\in [0,\phi(f(h))-h);
				\\
				\int_{0}^{r}
				\dfrac{1}{2} g^{\prime}(s)^2 \d s-f(h+r)\leqslant 0, 
				\\
				\quad\quad \quad \quad r\in [\phi(f(h))-h,1-h].
			\end{cases}
		\end{equation*}
		Then, we know
		\begin{equation}\label{eq21}
			\lim\limits_{n\rightarrow\infty}\dfrac{1}{a_n}\log\psi_n(h,g,\epsilon)=0.
		\end{equation}
		
		What's more,
		\begin{align}
			&P_n\left( \exists \nu\in N_n^n, 
			V_{\nu}^{n,a_n}(s)-V_\nu^{n,a_n}(h) >g_h(s)-\epsilon,
			\forall s\in [h,1] \right)
			\nonumber\\
			=&
			P_n\left(\bigcup\limits_{\gamma\in N^n_{\abs{nh}}} 
			\left\{
			\exists \nu\in N_n^n, \gamma \prec\nu
			V_{\nu}^{n,a_n}(s)-V_\nu^{n,a_n}(h) >g_h(s)-\epsilon,
			\forall s\in [h,1]
			\right\}
			\right)
			\nonumber\\
			=&
			E_n\left(1-\left(1-\psi_n(h,g,\epsilon)\right)^{Z^n_{\abs{nh}}} \right)
			\nonumber\\
			\geqslant&
			1-E_n\left(\exp\left(-Z^n_{\abs{nh}}\psi_n(h,g,\epsilon)\right)\right)\label{eq22}
		\end{align}
		
		Note that, for any $\beta >0$,
		\begin{align}
			E_n&\left(\exp\left(-Z^n_{\abs{nh}}\psi_n(h,g,\epsilon)\right)\right)
			=
			E_n\left(\exp\left(-\exp\left( 
			\log Z^n_{\abs{nh}} +\log\psi_n(h,g,\epsilon) \right) \right)\right)
			\nonumber\\
			&\leqslant
			P_n\left(Z^n_{\abs{nh}}\leqslant\exp\left(\beta S_{\abs{nh}}^n\right)\right)
			\nonumber\\
			&+E_n\left(\exp\left(-\exp\left( \log Z^n_{\abs{nh}} +
			\log\psi_n(h,g,\epsilon) \right) \right);
			Z^n_{\abs{nh}}>\exp\left(\beta S_{\abs{nh}}^n\right)\right)
			\nonumber\\
			&\leqslant
			P_n\left(Z^n_{\abs{nh}}\leqslant\exp\left(\beta S_{\abs{nh}}^n\right)\right)
			+
			\exp\left( -\exp\left( \beta S_{\abs{nh}}^n 
			+\log\psi_n(h,g,\epsilon) \right) \right)
			\label{eq23}
		\end{align}
		Under \Cref{A5}, we know
		\begin{equation}
			\lim\limits_{\beta\rightarrow0}
			\limsup\limits_{n\rightarrow\infty}
			P_n\left( Z^n_{\abs{nh}} \leqslant
			\exp\left( \beta S^n_{\abs{nh}} \right)\right)
			=0\label{eq24}
		\end{equation}
		Under \Cref{A2}, and \cref{eq21}, we know for any $\beta>0$,
		\begin{equation}\label{25}
			\lim\limits_{n\rightarrow\infty}E_n\left(
			\exp\left( -\exp\left( \beta S_{\abs{nh}}^n 
			+\log\psi_n(h,g,\epsilon) \right)
			\right)
			\right) = 0.
		\end{equation}
		
		Combine \cref{eq22,eq23,eq24,25},
		firstly let $n\rightarrow\infty$, then $\beta\rightarrow0$, we have
		\begin{align}
			\lim\limits_{n\rightarrow\infty}P_n\left( \exists \nu\in N_n^n,\right.&\left. 
			V_{\nu}^{n,a_n}(s)-V_\nu^{n,a_n}(h) >g_h(s)-\epsilon,
			\forall s\in [h,1] \right)
			\nonumber\\
			\geqslant&
			1-\lim\limits_{\beta\rightarrow0}\limsup\limits_{n\rightarrow\infty}
			E_n\left(\exp\left(-Z^n_{\abs{nh}}\psi_n(h,g,\epsilon)\right)\right)
			=1.\label{eq25}
		\end{align}
		
		About $\phi(f(h))$,
		\begin{align*}
			\left| g(\phi(f(h))) \right|^2&= \left| 
			\int_{0}^{\phi(f(h))} g^{\prime}(s)\d s\right|^2
			\\
			&\leqslant
			\left( \int_{0}^{\phi(f(h))} g^{\prime}(s)^2 \d s \right)
			\left( \int_{0}^{\phi(f(h))} 1 \d s \right)
			\\
			&\leqslant
			2f(h)\phi(f(h)).
		\end{align*}
		
		According to \Cref{Q and phi} and the continuity of $f$,
		then $\lim\limits_{h\rightarrow0^+} \phi(f(h)) = \phi(0^+)=0$,
		so there exist some $h_1>0$ such that $\phi(f(h))<\frac{1}{2}\epsilon$ for any $h\in (0,h_1)$. 
		
		According to \Cref{A2}, there exist some $h_2>0$ such that $f(h)< \epsilon$ for any $h\in (0,h_2)$.
		
		Together with the $h_0$ of \cref{eq20}, fix any $h\in (0,h_0\wedge h_1\wedge h_2)$, then $g(\phi(f(h))) < 2\epsilon$, 
		and recall that $g_h(1)= g(1)-g(\phi(f(h))) $, thus
		\begin{align}
			&P_n\left( \dfrac{M_n^{n}}{\sqrt{na_n}} \geqslant 
			g(1)-4\epsilon  \right)
			\geqslant P_n\left( \dfrac{M_n^{n}}{\sqrt{na_n}} \geqslant 
			g(1)-g(\phi(f(h)))- 2\epsilon \right)
			\nonumber\\
			\geqslant
			&P_n\left(  \exists \nu\in N_n^n, 
			V_{\nu}^{n,a_n}(s)-V_\nu^{n,a_n}(h) >g_h(s)-\epsilon,
			\right.
			\nonumber\\
			&\quad \quad\quad\quad\quad\quad\quad\quad\quad\quad\quad \left.
			\forall s\in [h,1]; 
			\forall \nu\in N_n^{n}, V_\nu^{n,a_n}(h)> -\epsilon 
			\right)
			\nonumber\\
			\geqslant
			&P_n\left(\exists \nu\in N_n^n, 
			V_{\nu}^{n,a_n}(s)-V_\nu^{n,a_n}(h)>g_h(s)-\epsilon,
			\forall s\in [h,1] \right)
			\nonumber\\
			&\quad \quad\quad\quad\quad\quad\quad\quad\quad\quad\quad
			+P_n\left( \forall \nu\in N_n^{n}, V_\nu^{n,a_n}(h) >-\epsilon \right)
			-1
			\label{eq26}
		\end{align}
		
		Combine \cref{eq20,eq25,eq26}, let $n\rightarrow\infty$, we have
		\begin{align}
			\lim\limits_{n\rightarrow\infty}
			P_n\left( \dfrac{M_n^{n}}{\sqrt{na_n}} \geqslant 
			g(1)-4\epsilon \right)
			=1.
		\end{align}
		
		Due to the arbitrariness of $\epsilon$, together with \cref{eq upper bound}, then
		\begin{equation*}
			\dfrac{M_n^n}{\sqrt{na_n}}\Longrightarrow A_f.
		\end{equation*}
		
		The proof of \Cref{Max Dis of TIBRW} is finished.
	\end{proof}
	
	\
	
	Now, we are going to prove \Cref{Q and phi}:
	\begin{proof}[Proof of \Cref{Q and phi}]
		It is easy to find that $g^\prime(s)\geqslant 0$ almost surely, which means $g$ is non-decreasing function (else, just set $g_1^{\prime}=|g^\prime|$, then the function $g_1$ also satisfies the condition in \cref{representation A_f} while $g_1(1)>g(1)=A_f$, which contradicts the definition of $A_f$). 
		
		Define $t_0:=\sup\{t\in [0,1];Q(t)=0 \} $, we will prove that $t_0=0$.
		If not, assume that $t_0>0$, then $g(t)=0$ for all $t\in [0,t_0]$ and 
		$Q(t)>0$ for $t\in(t_0,1]$.
		
		Let $R(t):= \dfrac{f(t)}{t-\frac{1}{2}t_0}$ in $t\in (\frac{t_0}{2},t_0]$ is a continuous and strict positive function, and $\lim_{t\rightarrow \frac{t_0}{2}^+}R(t)=\infty$ ($f(t)>0$ for all $t>0$).
		So there exist some $\gamma>0$ such that $R(t)>\gamma$ for all $t\in(\frac{t_0}{2},t_0]$.
		
		Note that, $f(t)-Q(t)$ is continuous function, 
		and $Q(t_0)=0$ and $f(t_0)>0$, then we can fix some $t_1\in (t_0,1)$, such that $f(t)-Q(t);t\in [t_0,t_1]\geqslant \dfrac{f(t_0)}{2}$ for all $t\in [t_0,t_1]$.
		
		Fix some $\delta\in (0, \sqrt{\frac{2f(t_0)}{t_0}}\wedge \sqrt{2\gamma}\wedge \dfrac{2Q(t_1)}{g(t_1)}\wedge \dfrac{2g(t_1)}{t_0})$.
		According to the choose of $\delta$,
		the set $\left( \dfrac{\delta^2t_0}{4Q(t_1)} , 
		\dfrac{\delta t_0}{2g(t_1)} \right)$ is non-empty,
		and then fix some $\alpha \in 
		\left( \dfrac{\delta^2t_0}{4Q(t_1)} , 
		\dfrac{\delta t_0}{2g(t_1)} \right) \subset (0,1)$. 
		Due to the choice of $\delta$ and $\alpha$,
		\begin{align*}
			&\frac{1}{2}\delta^2(t-\dfrac{t_0}{2}) < f(t), \quad \forall t\in [\frac{t_0}{2},t_0];\quad\quad \quad \quad  \alpha g(t_1)<\dfrac{\delta t_0}{2};
			\\
			&\dfrac{\delta^2t_0}{4}<\dfrac{f(t_0)}{2}< f(t)-Q(t), \quad \forall t\in [t_0,t_1];\quad
			\frac{\delta^2t_0}{4}<\alpha Q(t_1).
		\end{align*}
		
		Construct a function $g_1$ with derivative $g_1^\prime$ as follow:
		\begin{equation*}
			g_1^{\prime}(t):=
			\begin{cases}
				0,\quad t\in [0,\dfrac{t_0}{2}]
				\\
				\delta, \quad t\in(\dfrac{t_0}{2},t_0]
				\\
				(1-\alpha)g^{\prime}(t),\quad t\in ( t_0,t_1]
				\\
				g^{\prime}(t), \quad t\in (t_1,1].
			\end{cases}
		\end{equation*}
		Just calculate:
		\begin{equation*}
			\int_{0}^{t} \dfrac{1}{2} g_1^{\prime}(s)^2 \d s
			=
			\begin{cases}
				0 \leqslant f(t) ,\quad t\in [0,\frac{t_0}{2}]
				\\
				\frac{1}{2}\delta^2\left(t-\frac{1}{2}t_0\right)\leqslant f(t), 
				\quad t\in(\frac{t_0}{2},t_0]
				\\
				(1-\alpha)^2Q(t)+\frac{\delta^2t_0}{4}
				\\\quad\quad\quad\quad\quad\quad<(1-\alpha)^2Q(t)
				+f(t)-Q(t)\leqslant f(t),\quad t\in ( t_0,t_1]
				\\
				(1-\alpha)^2Q(t_1)+\frac{\delta^2t_0}{4}+Q(t)-Q(t_1)
				\\
				\quad\quad\quad\quad\quad\quad\leqslant
				Q(t)+\frac{\delta^2t_0}{4}-\alpha Q(t_1)
				<Q(t)\leqslant f(t), \quad t\in (t_1,1].
			\end{cases}
		\end{equation*}
		Thus $S_f(g_1)=0$ and
		\begin{align*}
			g_1(1)= \int_{0}^{1} g_1^{\prime} (s) \d s
			&=
			0+\dfrac{\delta t_0}{2}
			+(1-\alpha)g(t_1)
			+g(1)-g(t_1)\\
			&=
			g(1)+\dfrac{\delta t_0}{2}-\alpha g(t_1)>g(1)=A_f.
		\end{align*}
		which contradicts the definition of $A_f$.
		
		Thus $t_0=0$. Then for any $y>0$, $Q(y)>0$. And for the inverse function $\phi(h)$, due to the continuity of $Q(y)$, it is not hard to prove $\lim\limits_{h\rightarrow 0^+} \phi(h)= t_0=0$.
		
		The proof of \Cref{Q and phi} is over.
	\end{proof}

	\section{Critical branching processes in random environment}\label{Sec3}

	In this section, we will consider the critical branching processes in random environment. In order to use the results that established in \Cref{sec2} to help us investigate the conditional limit theorem, we just need to verify that conditioned on the survival events, the reduced branching processes do satisfy the assumptions in \Cref{Sec2.1}.
	
	\
	
	For an environment $\xi:=\{F_1,F_2,\cdots\}$,
	we also use $F_k$ denote the corresponding generating function.
	Define
	\[
	F_{k,n}(s):=
	\begin{cases}
		F_{k+1}\circ F_{k+2}\circ\cdots F_{n},\quad & k<n;
		\\
		\quad \quad \delta_1(s),\quad &k=n;
		\\
		F_{k}\circ F_{k+1}\circ\cdots F_{n+1}, \quad & k>n.
	\end{cases}
	\]
	Then $E_{\xi}\left[Z_n=0|Z_k=1\right]=F_{k,n}(0)$ 
	for any $k\leqslant n$, define $P_{\xi}(k,n):=1-F_{k,n}(0)$.
	
	For any $k\leqslant n $, 
	let $Z(k,n)$ be the number of particles at time $k$ 
	which have at least one descendant survived at time $n$.
	Given the environment $\xi$, 
	and conditional on $\{Z_n>0\}$,
	the reduced processes is a 
	time in-homogeneous branching process denoted by
	$\left\{Z(0,n),Z(1,n),\cdots,Z(n,n)|Z_n\right.$ $\left.>0\right\}$, 
	and its branching mechanism is given 
	by $\F^{r,n}=\{F^{r,n}_1,F^{r,n}_2,\cdots,F^{r,n}_n\}$ (here $r$  means ``reduced"), the generating function of $F_k^{r,n}$ has following representation (in fact, we can calculated it by the compound binomial distribution with condition):
	\begin{align}
		F^{r,n}_k(s):&=E_{\xi}[s^{Z(k,n)}|Z_{k-1}=1,Z_n>0]
		\nonumber\\
		&=\dfrac{1}{E_{\xi}[Z_n>0|Z_{k-1}=1]}
		\sum_{j=0}^{\infty} E_{\xi}[ s^{Z(k,n)} ;Z_{k}=j;Z(k,n)>0|Z_{k-1}=1].
		\nonumber\\
		&=\dfrac{ F_{k}\left(1-P_{\xi}(k,n)+sP_{\xi}(k,n)\right)
			-1+P_{\xi}(k-1,n) }{P_{\xi}(k-1,n)}\label{reduced branching mechanism}
	\end{align}
	Also see \cite{Fleischmann1977,Borovkov1997}.
	Recall \cref{mean and stand-var,quenched random walk}, take derivative in \cref{reduced branching mechanism}, we have
	\begin{equation}\label{reduce mean and sstand-var}
		\overline{F}^{r,n}_k:= \dfrac{P_{\xi}(k,n)}{P_{\xi}(k-1,n)}\exp(X_{k});\quad
		\tilde{F}^{r,n}_k:= P_{\xi}(k-1,n)\eta_{k}.
	\end{equation}
	
	The mean of the reduced processes is defined as $E_{\xi}[Z(k,n)|Z_n>0]$, 
	which depends only on the environment $\xi$ and satisfies
	\begin{equation}\label{mean of conditional reduced processes}
		E_{\xi}[Z(k,n)|Z_n>0]
		=
		\prod\limits_{j=1}^k \overline{F}_k^{r,n}
		=
		\dfrac{P_{\xi}(k,n)}{P_{\xi}(0,n)}\exp(S_k).
	\end{equation}
	Here $S_k$ is defined in \cref{quenched random walk}.
	
	\
	
	Define $L(k,n):=\min\{S_j:k\leqslant j\leqslant n\}$ 
	be the minimal value 
	of path $\{S_k,0\leqslant k \leqslant n\}$ between $k$ and $n$. Firstly, recall some  basic theorems describe the behavior of branching processes in random environment, which is related to the Brownian meander.
	
	\
	
	\begin{theorem}[Theorem 1.5 in \cite{Afanasyev2005}]\label{Environment Convergence to Brownian meander}
		Under \Cref{A1}, the following convergence holds in $D([0,1])$:
		\begin{equation*}
			\mathcal{L}\left(
			\left\{\dfrac{S_{\abs{nt}}}{\sigma\sqrt{n}};t\in[0,1]\right\}
			|Z_n>0\right)
			\Longrightarrow
			\mathcal{L}\left(
			\left\{ W_t^+;t\in [0,1] \right\}
			\right),\quad \text{as}\quad n\rightarrow\infty.
		\end{equation*}
		Here $ \left\{ W_t^+;t\in [0,1] \right\} $ is the Brownian meander.
	\end{theorem}
	
	\
	
	\begin{theorem}[Corollary 1.6 in \cite{Afanasyev2005}]\label{Branching Convergence to Brownian meander}
		Under \Cref{A1}, the following convergence holds in $D([0,1])$:
		\begin{equation*}
			\mathcal{L}\left(
			\left\{\dfrac{\log Z_{\abs{nt}}}{\sigma\sqrt{n}};t\in[0,1]\right\}
			|Z_n>0\right)
			\Longrightarrow
			\mathcal{L}\left(
			\left\{ W_t^+;t\in [0,1] \right\}
			\right),\quad \text{as}\quad n\rightarrow\infty.
		\end{equation*}
		Here $ \left\{ W_t^+;t\in [0,1] \right\} $ is the Brownian meander.
	\end{theorem}
	
	\
	
	\begin{remark}
		\Cref{Environment Convergence to Brownian meander} and \Cref{Branching Convergence to Brownian meander} is established for linear fractional case in \cite{Afanasyev1993}, \cite{Afanasyev1997} and many other articles for some wild cases. What's more, in \cite{Afanasyev2005}, Afanasyev, Geiger, Kersting and Vatutin investigated a new method, and extend such results to more general condition, especially under \Cref{A1}.
	\end{remark}
	
	\
	
	Now we are going to verified that for reduced branching processes in random environment, \Cref{A2} in \Cref{Sec2.1} is satisfied.
	\begin{lemma}\label{lemma of condition}
		Under \Cref{A1}, for any $\epsilon>0$,
		\begin{equation*}
			\lim\limits_{n\rightarrow\infty}
			\P\left( \log\left(1+\sum_{k=1}^{n}\eta_k \right) >
			\epsilon\sigma\sqrt{n} \left|\right. Z_n>0\right)=0.
		\end{equation*}
	\end{lemma}
	
	\begin{proof}[Proof of \Cref{lemma of condition}]
		For any $\epsilon>0$, define
		\begin{align*}
			W_n&:= \left\{ \max\limits_{1\leqslant k\leqslant n} |X_k|>\frac{\epsilon\sqrt{n}}{4} \right\}
			\nonumber\\
			T_n&:= \left\{ \max\limits_{1\leqslant k\leqslant n} \log_+ \varkappa_k(a) >\frac{\epsilon\sqrt{n}}{4}\right\}.
		\end{align*}
		Here $\varkappa_k(a)$ is defined in \Cref{A1}.
		Set $L_n:=\min\{S_0,S_1,\cdots,S_n\}$, note that \[\left\{\dfrac{S_{\abs{nt}}}{\sigma\sqrt{n}};t\in [0,1] |L_n\geqslant 0\right\} \dcon
		\left\{W_t^+;t\in [0,1]\right\}\] in $D([0,1])$, where $\left\{ W_t^+ \right\}$ is the Brownian meander, thus its path is continuous almost surely. So about the events $W_n$,
		\begin{equation}\label{Wn}
			\lim\limits_{n\rightarrow\infty}\P\left(W_n|L_n\geqslant0\right) = 0.
		\end{equation}

		About events $T_n$, we will use the structure that established in \cite{Afanasyev2005} to estimate the probability.
		Note that $\P^+$ denote the corresponding Doob-h transform of $\P$ in \cite{Afanasyev2005}, it is called the random walk condition to stay non-negative forever (see Section 2 in \cite{Afanasyev2005} for more details). According to the proof of Lemma 2.7 in \cite{Afanasyev2005}, there exist $\delta^\prime>0$, such that
		\begin{align*}
			\varkappa_k(a) = O\left( \exp\left(k^{\frac{1}{2}-\delta^\prime}\right) \right), \quad \P^+ a.s.
		\end{align*}
		Thus we know
		\begin{equation*}
			\ind{T_n} \rightarrow 0,\quad \P^+ a.s.
		\end{equation*}
		And due to dominating convergence theorem, above convergence also holds in mean. Due to Lemma 2.5 in \cite{Afanasyev2005},
		\begin{equation}
			\lim\limits_{n\rightarrow\infty}
			\P\left( T_n|L_n\geqslant 0 \right) = 0.\label{Tn}
		\end{equation}
		
		For any $k\geqslant1$, we have
		\begin{equation*}
			\eta_k\leqslant \varkappa_k(a)+a\exp(-X_k).
		\end{equation*}
		
		Note that, for fixed constant $\Gamma$, for $n$ large enough, such that $\exp\left(\frac{\epsilon\sigma\sqrt{n}}{2}\right)-\Gamma\geqslant an\geqslant n $, then
		\begin{align*}
			B_n(\Gamma)&:=\left\{
			\log\left(\Gamma+\sum_{k=1}^{n}\eta_k \right) >
			\epsilon\sigma\sqrt{n}
			\right\}
			\nonumber\\
			&\subset
			\left\{
			n\max\limits_{1\leqslant k\leqslant n} \varkappa_k+
			an\max\limits_{1\leqslant k\leqslant n} \exp(-X_k)>
			\exp\left(\epsilon\sigma\sqrt{n}\right)-1
			\right\}
			\nonumber\\
			&\subset
			\left\{\max\limits_{1\leqslant k\leqslant n} \varkappa_k+
			\max\limits_{1\leqslant k\leqslant n} \exp(-X_k) >
			\exp\left( \frac{\epsilon\sigma\sqrt{n}}{2}\right)
			\right\}
			\nonumber\\
			&\subset
			\left\{\max\limits_{1\leqslant k\leqslant n} \log_+\varkappa_k >
			\frac{\epsilon\sigma\sqrt{n}}{4} 
			\right\}
			\cup
			\left\{ \max\limits_{1\leqslant k\leqslant n} |X_k| >
			\frac{\epsilon\sigma\sqrt{n}}{4} 
			\right\}
		\end{align*}
		Hence, together with \cref{Wn,Tn}, we know for any constant $\Gamma$,
		\begin{equation*}
			\lim\limits_{n\rightarrow\infty}\P\left(B_n(\Gamma)|L_n\geqslant0\right)=0.
		\end{equation*}
		
		Define $\tau_n:=\min\left\{0\leqslant k\leqslant n;S_k= L_n\right\}$, 
		and  $ L_{k,n}:= \min\limits_{0\leqslant i\leqslant n-k}\left\{S_{k+i}-S_k\right\} $.
		Note that
		\begin{align}
			\P\left(B_n(1);Z_n>0\right)
			&=
			\sum_{k=0}^{n}
			\P\left(B_n(1);Z_n>0;\tau_k=k;L_{k,n}\geqslant 0\right)
			\nonumber\\
			&\leqslant
			\sum_{k=0}^{\infty}
			\P\left(B_n(1);Z_n>0;\tau_k=k;L_{k,n}\geqslant 0\right)\ind{k\leqslant n}
			\label{expand}
		\end{align}
		
		For $k\in \N$ fixed,
		define $ \Gamma_k:=  k+\sum\limits_{i=1}^k \eta_i<\infty $, it is a $\mathcal{F}_k$ measurable random variable, according to the dominating convergence theorem, and $\lim\limits_{n\rightarrow\infty}\sqrt{n}\P(L_n\geqslant 0) = K_1\in (0,\infty)$, then
		\begin{align*}
			&\lim\limits_{n\rightarrow\infty}
			\sqrt{n}\P\left(B_n(1);Z_n>0;\tau_k=k;L_{k,n}\geqslant 0\right)
			\nonumber\\
			\leqslant
			&\lim\limits_{n\rightarrow\infty}
			\sqrt{n}\E\left( \P\left(  \log\left( \Gamma_k+\sum_{i=k+1}^{n} \eta_i \right) >
			\epsilon\sigma\sqrt{n} |L_{k,n}\geqslant 0\right) ;\tau_k=k \right)
			\P(L_{n,k}\geqslant 0)\nonumber\\
			\leqslant
			&K_1\lim\limits_{n\rightarrow\infty}
			\E\left( \P\left(  B_{n-k}(\Gamma_k)|L_{n-k}\geqslant0 \right) \right)
			\nonumber\\
			\leqslant
			&K_1
			\E\left( \lim\limits_{n\rightarrow\infty}\P\left(  B_{n-k}(\Gamma_k)|L_{n-k}\geqslant0 \right) \right)
			=0.
		\end{align*}
		According to the dominating convergence theorem, combine \cref{Extinction Prob Decay,expand}, thus
		\begin{equation*}
			\lim\limits_{n\rightarrow\infty}\P(B_n(1)|Z_n>0)=0.
		\end{equation*}
		
		The proof of \Cref{lemma of condition} is over.
	\end{proof}
	
	\
	
	About the reduced branching processes in random environment, we have
	\begin{theorem}[Theorem 3 in \cite{Vatutin2002}]\label{Reduce branching convergence}
		Under \Cref{A1}, condition on $\{Z_n>0\}$, in $D([0,1])$, about the reduced branching processes in random environment $\{Z(k,n)\}$,
		\begin{equation*}
			\mathcal{L}\left(
			\left\{\dfrac{Z(\abs{nt},n)}{\sigma\sqrt{n}};t\in[0,1]\right\}
			|Z_n>0\right)
			\Longrightarrow
			\mathcal{L}\left(
			\left\{ \Lambda_t;t\in [0,1]\right\}
			\right),\quad \text{as}\quad n\rightarrow\infty.
		\end{equation*}
		Here $\Lambda_t:=\inf\{W_s^+;s\in [t,1]\}$, 
		where $ \left\{ W_t^+;t\in [0,1] \right\} $ is the Brownian meander.
	\end{theorem}
	
	\begin{remark}\label{Remark 11}
		In \cite{Vatutin2002}, Vatutin proved \Cref{Reduce branching convergence} under more stronger assumptions (comparing with \Cref{A1}). 
		At that time, the technology in \cite{Afanasyev2005} was not yet available, so the results about the random walk in random environment actually need more stronger assumptions. Together with the methods established in \cite{Afanasyev2005}, by repeat the steps in \cite{Vatutin2002}, we can expand related results to \Cref{A1}.
	\end{remark}
	
	\
	
	In this article, we need a more stronger results about the critical reduced branching processes in random environment.
	\begin{theorem}\label{Reduced convergence}
		Under the \Cref{A1}, 
		following convergence holds in $ D([0,1];\R^3)$:
		\begin{align*}
			\mathcal{L}\bigg(\left\{ \dfrac{\log E_{\xi}\left[ Z(\abs{nt},n) \left|\right. Z_n>0 \right]}{\sigma\sqrt{n}},\right.&\left.
			\dfrac{\log Z(\abs{nt},n)}{\sigma\sqrt{n}}, \dfrac{L(nt,n)}{\sigma\sqrt{n}} ;t\in[0,1] \left|\right. Z_n>0 \right\}\bigg)
			\\
			&\quad \Longrightarrow \mathcal{L}\left(\left\{\Lambda_t,\Lambda_t,\Lambda_t;t\in[0,1]\right\}\right)\quad \text{as}\quad n\rightarrow\infty.
		\end{align*} 
		Here $\Lambda_t:=\inf\{ W^+_s; s\in [t,1]\}$, 
		where $\left\{ W^+_t; t\in [0,1]\right\}$ is the Brownian meander.
	\end{theorem}

	\begin{remark}\label{reconstruce reduced branching processes}
		Condition on the survival events $\{Z_n>0\}$, 
		the reduced branching processes $\{Z(0,n),Z(1,n),\cdots,Z_(n,n)|Z_n>0\}$, 
		can be generated by following steps:
		\begin{itemize}
			\item Firstly, use $P_{\xi}(0,n)$ to bias the environment $\xi$, that is for any $A\subset \mathcal{P}(\N_0)^{\N}$ is a measurable set, define $ Q_n(A)= \dfrac{\P(P_{\xi}(0,n);A)}{\P(Z_n>0)}$, then $Q_n$ is a probability on $\mathcal{P}(\N_0)^{\N}$.
			According to law $ Q_n $, generate an environment $\xi$.
			\item Given the environment $\xi$,
			get the reduced branching mechanism 
			$\F^{r,n}(\xi):=\{F_1^{r.n}(\xi),\cdots,$ $F_k^{r,n}(\xi),\cdots,F_n^{r,n}(\xi)\}$ (see \cref{reduced branching mechanism}).
			\item Then, drive a branching process 
			according to the reduced branching mechanism $\F^{r,n}(\xi)$ 
			until time $n$, and the process has the same law as the reduced branching processes.
		\end{itemize}
	\end{remark}

	\begin{proof}[Proof of \Cref{Reduced convergence}]
		For $\epsilon>0$, define
		\begin{align*}
			&H_n(\epsilon):=
			\left\{
			\exists t\in [0,1],
			\left| \dfrac{\log E_{\xi}\left[ Z(\abs{nt},n) 
				\left|\right. Z_n>0 \right]   
				-L(\abs{nt},n)}
			{\sigma\sqrt{n}} \right|>\epsilon
			\right\}
			\nonumber\\
			&R_n(\epsilon):=
			\left\{\exists t\in [0,1],
			\left| \dfrac{\log Z\left(\abs{nt},n\right)-L(\abs{nt},n)}
			{\sigma\sqrt{n}} \right|>\epsilon
			\right\}
		\end{align*}
		
		\begin{lemma}\label{Sn convergence}
			For any $\epsilon>0$, the event $H_n(\epsilon)$, satisfies
			$ \lim\limits_{n\rightarrow\infty}
			\P\left(H_n(\epsilon)\left|\right.Z_n>0\right)
			=0 $
		\end{lemma}
		
		\begin{proof}[Proof of \cref{Sn convergence}]
			Note that for any $k\leqslant l\leqslant n$, 
			\begin{align*}
				P_{\xi}(k,n)\leqslant P_{\xi}(k,l)
				\leqslant E_{\xi}(Z_l|Z_k=1) = \exp(S_l-S_k).
			\end{align*}
			Hence, we have
			\begin{align}\label{Prob Upper}
				\log P_{\xi}(k,n)\leqslant \min\{S_l-S_k; k\leqslant l\leqslant n\} = L(k,n)-S_k.
			\end{align}
			and due to Agresti?s estimate (see the Section 2 in \cite{Geiger2000}, or eq (3.4) in \cite{Afanasyev2005}), we have
			\begin{equation}\label{Prob lower}
				\log P_{\xi}(k,n)
				\geqslant
				L(k,n)-S_k-\log\left(1+\sum_{j=k+1}^{n}\eta_j\right)
			\end{equation}
			Combine \cref{Prob Upper,Prob lower} and $S_0=0$,
			\begin{align*}
				-L(0,n)-\log\left(1+\sum_{j=1}^n \eta_j\right).
				&\leqslant
				-\log P_{\xi}(0,n)-\log\left(1+\sum_{j=k+1}^n\eta_j\right)
				\nonumber\\
				&\leqslant
				-\log P_\xi(0,n)+\log P_\xi(k,n)+S_k-L(k,n)
				\nonumber\\
				&=
				\log E_{\xi}\left[ Z(\abs{nt},n) \left|\right. Z_n>0 \right] -L(k,n)
				\\
				&\leqslant
				-\log P_\xi(0,n)
				\leqslant
				-L(0,n)+\log\left(1+\sum_{j=1}^n \eta_j\right).
			\end{align*}
			Note that $L(0,n)\leqslant S_0=0$, 
			and $\log(1+\sum\limits_{j=1}^n \eta_j)\geqslant0$, therefore
			\begin{align*}
				\left| \dfrac{\log E_{\xi}\left[Z\left(\abs{nt},n\right)
					\left|\right.Z_n>0\right]}{\sigma\sqrt{n}}-
				\dfrac{L(\abs{nt},n)}
				{\sigma\sqrt{n}} \right|
				\leqslant
				\dfrac{\left|L(0,n)\right|+\log\left(1+\sum\limits_{k=1}^n\eta_k\right)}
				{\sigma\sqrt{n}}
			\end{align*}
			Hence,
			\begin{align*}
				H_n(\epsilon)
				&\subset
				\left\{
				\dfrac{\left|L(0,n)\right|+\log\left(1+\sum\limits_{k=1}^n\eta_k\right)}
				{\sigma\sqrt{n}}>\epsilon
				\right\}
				\\
				&\subset
				\left\{
				\left|L(0,n)\right|\geqslant \dfrac{\epsilon\sigma\sqrt{n}}{2}
				\right\}
				\cup
				\left\{
				\log\left(1+\sum\limits_{k=1}^n\eta_k\right)>
				\dfrac{\epsilon\sigma\sqrt{n}}{2}
				\right\}
			\end{align*}
			
			According to \Cref{Environment Convergence to Brownian meander} and \Cref{lemma of condition}, then \Cref{Sn convergence} is proved.
		\end{proof}
		
		\begin{lemma}\label{R_n convergence}
			For any $\epsilon>0$, the event $R_n(\epsilon)$, satisfies
			$ \lim\limits_{n\rightarrow\infty}
			\P\left(R_n(\epsilon)\left|\right.Z_n>0\right)
			= 0. $
		\end{lemma}
		
		\begin{proof}[Proof of \Cref{R_n convergence}]
			Denote $\chi_n(t):= \dfrac{\log
				Z(\abs{nt},n)-L(\abs{nt},n)}{\sigma\sqrt{n}} $
			
			Firstly, we will prove that for any fixed $t\in [0,1]$, 
			\begin{equation}\label{Fixed t}
				\lim\limits_{n\rightarrow\infty}
				\P\left(|\chi_n(t)|>\epsilon|Z_n>0\right)=0.
			\end{equation}
			
			Define:
			\begin{align*}
				A_n:=
				\left\{
				\exp\left(S_{\abs{nt}}-\frac{\epsilon}{2}\sigma\sqrt{n}\right)
				\leqslant Z_{\abs{nt}}\leqslant 
				\exp\left(S_{\abs{nt}}+\frac{\epsilon}{2}\sigma\sqrt{n}\right)
				\right\}
			\end{align*}
			Due to \eqref{Prob Upper},
			on the event $A_n\cap \{\chi_n(t)>\epsilon\}$,
			\begin{align*}
				Z\left(\abs{nt},n\right)
				&> \exp\left(L(\abs{nt},n)+\epsilon\sigma\sqrt{n}\right)
				\geqslant
				Z_{\abs{nt}}P_{\xi}(\abs{nt},n)
				\exp\left(\frac{\epsilon\sigma\sqrt{n}}{2}\right)
			\end{align*}
			Then,
			\begin{align*}
				\P\left(
				\ind{\chi_{n}>\epsilon}\ind{A_n}|Z_n>0\right)\leqslant
				\P\left(A_n\cap C_n|Z_n>0\right)
			\end{align*}
			Here $ C_n:=\{ Z\left(\abs{nt},n\right)
			>Z_{\abs{nt}}P_{\xi}(\abs{nt},n)
			\exp\left(\frac{\epsilon\sigma\sqrt{n}}{2}\right) \}$.
			
			Note that, condition on $\mathcal{F}_{\abs{nt}}$ 
			(here $\mathcal{F}_k$ denote as the $\sigma$ algebra generated by total
			environment and the branching processes up to time $k$),
			$Z(\abs{nt},t)$ can be view as the binomial distribution with 
			parameter $\left(Z_{\abs{nt}}, P_{\xi}(\abs{nt},n)\right)$.
			So,
			\begin{align*}
				\P
				\left(
				\ind{A_n}\ind{C_n}\ind{Z_n>0}
				\right)
				&=
				\P\left(\P\left(
				\ind{A_n}\ind{C_n}
				\ind{Z_n>0}
				|\mathcal{F}_{\abs{nt}}\right)\right)
				\\
				&=
				\P\left( \ind{A_n} \P\left(
				\ind{C_n}
				|\mathcal{F}_{\abs{nt}}\right)\right)
				\\
				&\leqslant
				\P\left(
				\ind{A_n}
				\exp\left(-\frac{\epsilon\sigma\sqrt{n}}{2}\right)
				\right)
				\leqslant
				\exp\left(-\frac{\epsilon\sigma\sqrt{n}}{2}\right)
			\end{align*}
			And 
			due to \Cref{Extinction Prob Decay},
			and $\lim\limits_{n\rightarrow\infty}
			\sqrt{n}\exp\left(-\frac{\epsilon\sigma\sqrt{n}}{2}\right)=0$, 
			thus
			\begin{equation*}
				\lim\limits_{n\rightarrow\infty} \P(A_n\cap C_n|Z_n>0)=0
			\end{equation*}
			
			According to \Cref{Branching Convergence to Brownian meander}, we know
			$ \lim\limits_{n\rightarrow\infty}\P(A_n|Z_n>0)=1.$
			Therefore, for any $\epsilon>0$,
			\begin{equation}\label{Upper is empty}
				\lim\limits_{n\rightarrow\infty}\P(\chi_n(t)>\epsilon|Z_n>0)=0.
			\end{equation}
			
			According to \Cref{Environment Convergence to Brownian meander} and \Cref{Reduce branching convergence}, for any $x>0$,
			\begin{align}
				\lim\limits_{n\rightarrow\infty}
				\P&\left(
				\dfrac{\log Z\left(\abs{nt},n\right)}{\sigma\sqrt{n}}>x|Z_n>0
				\right)
				=
				\lim\limits_{n\rightarrow\infty}
				\P\left(
				\chi_n(t)
				+
				\dfrac{L(\abs{nt},n)}{\sigma\sqrt{n}}
				>x
				|Z_n>0\right)
				\nonumber\\
				&\quad\quad\quad\quad\quad\quad\quad\quad=
				\lim\limits_{n\rightarrow\infty}
				\P\left(
				\dfrac{L(\abs{nt},n)}{\sigma\sqrt{n}}
				>x
				|Z_n>0\right)
				=
				\P\left(
				\Lambda_t>x
				\right).\label{same limit}
			\end{align}
			Here $\Lambda_t:=\inf\{ W^+_s; s\in [t,1]\}$, 
			where $\left(W^+_t; t\in [0,1]\right)$ is a Brownian meander.
			
			Combine \cref{Upper is empty,same limit}, we will prove 
			$ \lim\limits_{n\rightarrow\infty} 
			\P\left(\chi_n(t)<-\epsilon|Z_n>0\right)=0 $.

			Firstly, for any $\delta>0$,
			there exist $\gamma>0$ and $K$, such that for $n$ large enough,
			\begin{align}
				\P\left(
				\dfrac{L(\abs{nt},n)}{\sigma\sqrt{n}}<\gamma
				|Z_n>0
				\right)
				+
				\P\left(
				\dfrac{L(\abs{nt},n)}{\sigma\sqrt{n}}>\gamma+K\epsilon
				|Z_n>0
				\right)<\delta.\label{63}
			\end{align}
			Note that
			\begin{align}
				\P\left(\chi_n(t)\right.&\left.<-\epsilon|Z_n>0\right)
				=
				\P\left(
				\chi_n(t)<-\epsilon,
				\dfrac{L(\abs{nt},n)}{\sigma\sqrt{n}}<\gamma
				|Z_n>0
				\right)
				\nonumber\\
				&+
				\P\left(
				\chi_n(t)<-\epsilon,
				\dfrac{L(\abs{nt},n)}{\sigma\sqrt{n}}\geqslant\gamma+K\epsilon
				|Z_n>0
				\right)
				\nonumber\\
				&+
				\sum_{k=1}^K
				\P\left(
				\chi_n(t)<-\epsilon,
				\dfrac{L(\abs{nt},n)}{\sigma\sqrt{n}}\in 
				\left[ \gamma+(k-1)\epsilon,\gamma+k\epsilon \right)
				|Z_n>0
				\right)\label{64}
			\end{align}
			
			For any $y>\epsilon$, and $\epsilon_1>0$,
			\begin{align*}
				&~\quad\P\left(
				\chi_n(t)
				+
				\dfrac{L(\abs{nt},n)}{\sigma\sqrt{n}}
				>y
				|Z_n>0\right)
				\nonumber\\
				&=
				\P\left(
				\chi_n(t)
				+
				\dfrac{L(\abs{nt},n)}{\sigma\sqrt{n}}
				>y;\chi_n(t)<-\epsilon
				|Z_n>0\right)
				\nonumber\\
				&+
				\P\left(
				\chi_n(t)
				+
				\dfrac{L(\abs{nt},n)}{\sigma\sqrt{n}}
				>y;\chi_n(t)\in[-\epsilon,\epsilon_1]
				|Z_n>0\right)
				\nonumber\\
				&+
				\P\left(
				\chi_n(t)
				+
				\dfrac{L(\abs{nt},n)}{\sigma\sqrt{n}}
				>y;\chi_n(t)>\epsilon_1
				|Z_n>0\right)
				\nonumber\\
				&\leqslant
				\P\left(
				\dfrac{L(\abs{nt},n)}{\sigma\sqrt{n}}
				>y+\epsilon;\chi_n(t)<-\epsilon
				|Z_n>0\right)
				\nonumber\\
				&+
				\P\left(\dfrac{L(\abs{nt},n)}{\sigma\sqrt{n}}
				>y-\epsilon_1,
				\chi_n(t)\geqslant-\epsilon_1
				|Z_n>0 \right)
				\nonumber\\
				&+
				\P\left(
				\chi_n(t)>\epsilon_1
				|Z_n>0\right)
			\end{align*}
			Consequently,
			\begin{align}
				&\P\left(
				\chi_n(t) +\dfrac{L(\abs{nt},n)}{\sigma\sqrt{n}}
				>y |Z_n>0\right)
				\leqslant
				\P\left( \dfrac{L(\abs{nt},n)}{\sigma\sqrt{n}} >y-\epsilon_1|Z_n>0\right)
				\nonumber\\
				&\quad\quad\quad\quad\quad\quad\quad\quad\quad-\P\left(\dfrac{L(\abs{nt},n)}{\sigma\sqrt{n}}\in [y,y+\epsilon],\chi_n(t)<-\epsilon |Z_n>0 \right)
				\nonumber\\
				&\quad\quad\quad\quad\quad\quad\quad\quad\quad
				+\P\left(
				\chi_n(t)>\epsilon_1
				|Z_n>0\right)\label{take limit}
			\end{align}
			
			Let $n_k$, be any increasing sequence of integers such that the limit
			\[
			c=
			\lim\limits_{k\rightarrow\infty}
			\P\left(\dfrac{L(\abs{n_kt},n_k)}{\sigma\sqrt{n_k}}
			\in [y,y+\epsilon],
			\chi_{n_k}(t)<-\epsilon
			|Z_{n_k}>0 \right)
			\]
			exists. We will show $c=0$. Indeed,
			passing to the limit, as $n_k\rightarrow\infty$ 
			in both sides of inequality \cref{take limit},
			together with 
			\cref{Upper is empty,same limit},
			then
			\begin{equation*}
				\P\left(
				\Lambda_t>y
				\right)
				\leqslant
				\P\left(
				\Lambda_t>y-\epsilon_1
				\right)-c
			\end{equation*}
			let $\epsilon_1$ tend to $0$, then
			\[
			\P\left(
			\Lambda_t>y
			\right)
			\leqslant
			\P\left(
			\Lambda_t>y \right)-c
			\]
			What's more, $c$ is non-negative. Then $c=0$.
			Due to the abriatily of $n_k$, we know
			\begin{equation}
				\lim\limits_{n\rightarrow\infty}
				\P\left(\dfrac{L(\abs{nt},n)}{\sigma\sqrt{n}}
				\in [y,y+\epsilon],
				\chi_{n}(t)<-\epsilon
				|Z_{n}>0 \right) = 0.
			\end{equation}
			Thus, for each $k$,
			\begin{equation}\label{67}
				\lim\limits_{n\rightarrow\infty}
				\P\left(
				\chi_n(t)<-\epsilon,
				\dfrac{L(\abs{nt},n)}{\sigma\sqrt{n}}\in 
				\left[ \gamma+(k-1)\epsilon,\gamma+k\epsilon \right)
				|Z_n>0
				\right)=0.
			\end{equation}
			Then, combine \cref{63,64,67}, we know
			$ \lim\limits_{n\rightarrow\infty}
			\P
			\left(
			\chi_{n}(t)<-\epsilon
			|Z_n>0 \right)\leqslant \delta. $
			Due to the arbitrariness of $\delta$,
			\begin{equation*}
				\lim\limits_{n\rightarrow\infty}
				\P
				\left(
				\chi_{n}(t)<-\epsilon
				|Z_n>0 \right) = 0.
			\end{equation*}
			Hence \cref{Fixed t} is proved.
			
			Fix any $m\in \N$, define
			\begin{align*}
				D_m(\epsilon,n):=
				&\bigcup\limits_{k=1}^m
				\left\{\left|
				\dfrac{\log Z(\abs{n\frac{k}{m}},n)-L(\abs{n\frac{k-1}{m}},n)}{\sigma\sqrt{n}}
				\right|>\epsilon
				\right\}
				\\
				&\quad\quad\quad\quad\bigcup\limits_{k=1}^m \left\{ \left| \dfrac{\log Z(\abs{n\frac{k-1}{m}},n)-L(\abs{n\frac{k}{m}},n)}{\sigma\sqrt{n}}
				\right|>\epsilon \right\}
			\end{align*}
			According to the monotonicity of $Z(k,n)$ and $L(k,n)$, 
			so $ R_n(\epsilon)\subset D_m(\epsilon,n) $.
			For \cref{Fixed t}, take $t=0,\frac{1}{m},\cdots,\frac{m}{m}$, combine with \Cref{Branching Convergence to Brownian meander}, we know
			\begin{align}
				\lim\limits_{n\rightarrow\infty}
				\P\left( D_m(\epsilon,n) \right)
				=
				\P\left(
				\bigcup\limits_{k=1}^m 
				\left\{ \Lambda_{\frac{k}{m}} -\Lambda_{\frac{k-1}{m}}>\epsilon \right\}
				\right)\label{70}
			\end{align}
			Here $\Lambda_t:=\inf\{ W^+_s; s\in [t,1]\}$, 
			where $\left(W^+_t; t\in [0,1]\right)$ is a Brownian meander.
			
			Note that the path of Brownian meander is continuous, then
			\begin{align}
				\lim\limits_{m\rightarrow\infty}
				\P\left(
				\bigcup\limits_{k=1}^m 
				\left\{ \Lambda_{\frac{k}{m}} -\Lambda_{\frac{k-1}{m}}>\epsilon \right\}
				\right)=0.\label{71}
			\end{align}
			Combine \cref{70,71}, $R_n(\epsilon)\subset D_m(\epsilon,n) $ and the arbitrarily of $m$,
			\begin{equation*}
				\lim\limits_{n\rightarrow\infty}
				\P\left(R_n(\epsilon)|Z_n>0\right)
				=0.
			\end{equation*}
			
			The proof of \Cref{R_n convergence} is over.
		\end{proof}
		
		According to \Cref{lemma of condition}, \Cref{Sn convergence}  and  \Cref{R_n convergence}, together with \Cref{Reduce branching convergence}, then \Cref{Reduced convergence} is proved.
	\end{proof}

	\section{Proof of Theorems}\label{Sec4}
	In this section, we will prove the \Cref{Conditional limit theorem} and \Cref{Tail prob of M}.
	
	\subsection{Proof of \Cref{Conditional limit theorem}}\label{Sec4.1}
	In this subsection, 
	We will use the coupling methods to prove 
	\Cref{Conditional limit theorem}.
	Due to \cref{reduce mean and sstand-var}, \Cref{lemma of condition}, \Cref{Reduced convergence} 
	and \Cref{reconstruce reduced branching processes},
	use the Skorohod representation theorem,
	we can construct a probability space 
	$\left( \Omega,\mathcal{F}, \Q \right)$, 
	such that
	\begin{itemize}
		\item There is a Brownian meander $\left\{ W^+_t;t\in [0,1] \right\}$, 
		and let $\Lambda_t:=\inf\{ W^+_s; s\in [t,1]\}$.
		\item For each $n$, there exist an environment $\xi_n = \left(F_1^n,F_2^n,\cdots\right)\in \mathcal{P}(\N_0)^{\N}$, 
		and the law of $\xi_n$ is just $Q_n$ (that is $\mathcal{L}(\xi|Z_n>0)$, see \Cref{reconstruce reduced branching processes}).
		And for each $\xi_n$, according to formula \cref{reduced branching mechanism}, we can get the reduced branching mechanism
		$ \F^{r,n} $.
		\item Let $X_k^n:=\overline{F}_k^n$ (see \cref{reduce mean and sstand-var}), $S_0^n:=0 $ 
		and $ S_k^n:=S_{k-1}^n+X_k^n $, in fact we have $S_k^n = \log E_{\xi}[Z(k,n)|Z_n>0] $ (see \cref{mean of conditional reduced processes}), according to \Cref{Reduced convergence}, use Skorohod representation theorem, we have
		\begin{equation}\label{One of A2}
			\lim\limits_{n\rightarrow\infty}
			\dfrac{S^n_{\abs{nt}}}{\sigma\sqrt{n}}= \Lambda_t, \quad \Q \quad a.s.
		\end{equation}
		\item Due to \cref{reduce mean and sstand-var}, we know $\tilde{F}^{r,n}_k \leqslant \eta_k^n = \tilde{F}_k^n $, and according to \Cref{lemma of condition}, use Skorohod representation theorem, we have
		\begin{equation}\label{Two of A2}
			\lim\limits_{n\rightarrow\infty}
			\dfrac{1}{\sigma\sqrt{n}}\log\left(1+\sum_{k=1}^n\tilde{F}_k^{r,n}\right)
			=0,\quad \Q \quad a.s..
		\end{equation}
		\item Conditional on all above, for each $n$, 
		there exist a branching random walk $(\T^n,V^n)$, 
		whose branching mechanism is just $\F^{r,n}$ 
		and the jump distribution is just $\ND(0,1)$.
		Let $M_n^n$ 
		be the maximal displacement of the branching random walk at time $n$.
		\item Use $Z_k^n$ denote the number of particles
		survived at generation $k$, according to \Cref{Reduced convergence}, use Skorohod representation theorem, we have for each $t\in [0,1]$, 
		\begin{equation}\label{One of A5}
			\lim\limits_{n\rightarrow\infty}
			\dfrac{\log Z^n_{\abs{nt}}}{\sigma\sqrt{n}}= \Lambda_t,\quad \Q\quad a.s..
		\end{equation}
	\end{itemize}
	
	Define $\mathcal{G}:= \sigma\{W_t^+;t\in[0,1]\}
	\cup\{\xi_n;n\geqslant 0\}$ and
	$A_\Lambda:=\sup\{g(1):\int_{0}^{r}\frac{1}{2}g^{\prime}(s)^2\d s\leqslant \Lambda_r\}$. Then  $A_\Lambda$ is $\mathcal{G}$ measurable, and $\mathcal{G}$ is the $\sigma$-field generated be the environment.
	
	In order to use the \Cref{Max Dis of TIBRW}, we should verify condition on $\mathcal{G}$, the \Cref{A2} and \Cref{A5} are satisfied for above time-inhomogeneous branching random walk.
	
	In fact, take $a_n = \sigma\sqrt{n}$ and $ f(t)= \Lambda_t$, 
	\cref{One of A2,Two of A2} ensure the \Cref{A2} is satisfied almost surely, and \cref{One of A5} ensures the \Cref{A5} is satisfied.
	Then, according to \Cref{Max Dis of TIBRW}, for any $\epsilon>0$, we have
	\begin{equation*}
		\lim\limits_{n\rightarrow\infty}
		\Q
		\left(\left|\dfrac{M_n^n}{\sqrt{\sigma}n^{\frac{3}{4}}}-A_\Lambda(\omega)\right|>\epsilon|\mathcal{G}\right)
		= 0, \quad \Q \quad a.s..
	\end{equation*}
	According to dominated convergence theorem, we have
	\begin{align*}
		\lim\limits_{n\rightarrow\infty}
		\Q\left( \left|\dfrac{M_n^n}{\sqrt{\sigma}n^{\frac{3}{4}}} -A_\Lambda(\omega)\right|>\epsilon \right)
		=\Q
		\left( \lim\limits_{n\rightarrow\infty} 
		\Q\left( \left|\dfrac{M_n^n}{\sqrt{\sigma}n^{\frac{3}{4}}} -A_\Lambda(\omega)\right|>\epsilon |\mathcal{G}\right)
		\right)
		=0.
	\end{align*}
	Due to the arbitrariness of $\epsilon$, we know $ \dfrac{M^n_n}{\sqrt{\sigma}n^{\frac{3}{4}}} \Longrightarrow A_\Lambda $ under $\Q$.
	
	Note that the law of $M_n^n$ under $\Q$ has is exactly $\mathcal{L}\left(M_n^n|Z_n>0\right)$.
	Combine above relationship, we have
	\begin{equation*}
		\mathcal{L}\left( \dfrac{M_n^n}{\sqrt{\sigma}n^{\frac{3}{4}}}\big|Z_n>0\right) 
		\Longrightarrow \mathcal{L}(A_\Lambda),\quad \text{as}\quad n\rightarrow\infty.
	\end{equation*}
	
	Hence \Cref{Conditional limit theorem} is proved.
	
	\subsection{Proof of \Cref{Tail prob of M}}\label{Section4.2}
	In this subsection, we will prove \Cref{Tail prob of M}.
	During this subsection, $K$ represent the same constant in \cref{Extinction Prob Decay}.
	Firstly, according to \Cref{Conditional limit theorem}, we know
	\begin{align*}
		\liminf\limits_{n\rightarrow\infty}\sqrt{n}\P(M>n^{\frac{3}{4}}) 
		&\geqslant
		\liminf\limits_{n\rightarrow\infty}\sqrt{n}
		\P(M_n>n^{\frac{3}{4}},Z_n>0)
		\nonumber\\
		&\geqslant
		\liminf\limits_{n\rightarrow\infty}\sqrt{n} \P\left(Z_n>0\right)
		\P\left( M_n>n^{\frac{3}{4}} \left|\right.Z_n>0\right)
		\nonumber\\
		&\geqslant K \P\left(\sqrt{\sigma}A_\Lambda>1\right).
	\end{align*}
	
	\begin{lemma}\label{Prop of AM}
		About $ A_\Lambda$, for any $x>0$,
		\begin{equation*}
			\P(A_\Lambda\geqslant x) \geqslant \P\left( \Lambda_{\frac{1}{2}} \geqslant x^2 \right) > 0.
		\end{equation*}
		Here $\Lambda_t:=\inf\{ W^+_s; s\in [t,1]\}$, 
		where $\left(W^+_t; t\in [0,1]\right)$ is a Brownian meander.
	\end{lemma}
	
	\begin{proof}[Proof of \Cref{Prop of AM}]
		Note that one the event $\{\Lambda_{\frac{1}{2}}\geqslant x^2\}$, construct function 
		\[g(r):=\begin{cases}
			0, \quad &r\in [0,\frac{1}{2}];
			\\
			2xr-x,&r\in [\frac{1}{2},1].
		\end{cases}\]
		By calculating, we know $\int_{0}^r \frac{1}{2}g^{\prime}(s)^2 \d s=\begin{cases}
			0,\quad &r\in [0,\frac{1}{2}];
			\\
			2x^2r-x^2,&r\in [\frac{1}{2},1].
		\end{cases}$, then $ \int_{0}^r \frac{1}{2}g^{\prime}(s)^2 \d s \leqslant \Lambda_r $ for any $r\in [0,1]$. Therefore, according to the definition of $A_\Lambda$, we know that $A_\Lambda \geqslant g(1)=x$.
		Then  \[ \P(A_\Lambda\geqslant x) \geqslant \P\left( \Lambda_{\frac{1}{2}} \geqslant x^2 \right).\]
		According to the property of Brownian meander, we know that $ \P\left(\Lambda_{\frac{1}{2}} \geqslant x^2\right)>0 $.
		Then \Cref{Prop of AM} is proved.
	\end{proof}
	
	According to \Cref{Prop of AM}, 
	we know  $\P\left(\sqrt{\sigma}A_\Lambda>1\right)>0$. For $C_1= K \P\left(\sqrt{\sigma}A_\Lambda>1\right) >0$, we have
	\begin{equation*}
		\liminf\limits_{x\rightarrow\infty}x^{\frac{2}{3}}\P(M>x)\geqslant C_1.
	\end{equation*}
	The left hand of \cref{tail estimate} is proved.
	
	\

	For the right hand of \cref{tail estimate}, 
	choose $\epsilon>0$,
	\begin{align}
		\P\left(M>n^{\frac{3}{4}} \right)
		&=
		\P\left(M>n^{\frac{3}{4}};Z_{\epsilon n}>0\right)
		+
		\P\left(M>n^{\frac{3}{4}};Z_{\epsilon n}=0\right)
		\nonumber\\
		&=
		\P\left(Z_{\epsilon n}>0\right)
		\P\left(M>n^{\frac{3}{4}}\left|\right.Z_{\epsilon n}>0\right)
		+
		\P\left(M>n^{\frac{3}{4}};Z_{\epsilon n}=0\right)\label{eq44}
	\end{align}
	
	Note that 
	\begin{align}
		\P\left(M>n^{\frac{3}{4}}\right| &\left.Z_{\epsilon n}>0\right)
		\leqslant
		\P\left(\sup\limits_{t\in [0,1]}\log\left(Z_{tT}+1\right)> \sigma\sqrt{n} \left|\right.Z_{\epsilon n}>0\right)
		\nonumber\\
		&+
		\P\left(T> \frac{n}{4} \left|\right.Z_{\epsilon n}>0\right)
		\nonumber\\
		&+
		\P\left(M>n^{\frac{3}{4}};
		\sup\limits_{t\in [0,1]}\log\left(Z_{tT}+1\right)> \sigma\sqrt{n};
		T\leqslant \frac{n}{4} \left|\right.Z_{\epsilon n}>0\right)
		\label{eq45}
	\end{align}
	Here $T:=\inf\left\{k:Z_k=0\right\}$ is the time of extinction.
	
	According to \cref{Extinction Prob Decay}, we know
	\begin{equation}\label{eq46}
		\lim\limits_{n\rightarrow\infty}
		\P\left(T>\frac{n}{4}|Z_{\epsilon n}>0\right)=
		\lim\limits_{n\rightarrow\infty}
		\dfrac{\P\left(T>\frac{n}{4}\right)}{\P\left(T>\epsilon n\right)} = 2\sqrt{\epsilon}
	\end{equation}
	
	The following theorem was proved by Afanansyev in \cite{Afanasyev1999},
	\begin{theorem}[Theorem 5 in \cite{Afanasyev1999}]\label{L2 of CBPRE}
		Under \Cref{A1},
		\begin{equation*}
			\mathcal{L}\left(\left\{ \dfrac{\log\left(Z_{tT}+1\right)}{\sigma\sqrt{n}}\big|T>n\right\}\right)
			\Longrightarrow
			\mathcal{L}\left(\left\{ \dfrac{W_0^+(t)}{\alpha};t\in [0,1] \right\}\right),\quad \text{as}\quad n\rightarrow\infty.
		\end{equation*}
		Here $W_0^+(t)$ is a normal Brownian excursion and $\alpha$ is a independent random variable uniformly distribution on $(0,1)$.
	\end{theorem}
	
	\begin{remark}
		Compare with the Theorem 5 in \cite{Afanasyev1999}, the assumption in \Cref{L2 of CBPRE} is much weaker. At that time, the assumptions in \cite{Afanasyev1999} is just to ensure the \Cref{Branching Convergence to Brownian meander} was correct. As the same argument in \Cref{Remark 11}, we can repeat the proof in \cite{Afanasyev1999} to extend related results for \Cref{A1}.
	\end{remark}
	
	\
	
	Repeat the arguments of eqs. $(37)$ to $(39)$ of \cite{Afanasyev1999}, we have
	\begin{align}
		\lim\limits_{n\rightarrow\infty}
		\P\left(\sup\limits_{t\in [0,1]}\log\left(Z_{tT}+1\right)> \sqrt{n} \left|\right.Z_{\epsilon n}>0\right)
		&=
		\sqrt{\epsilon}\sigma\int_{0}^{\frac{1}{\sqrt{\epsilon}}}
		\P\left(\sup\limits_{t\in [0,1]} W^+_0(t)> x\right) \d x
		\nonumber\\
		&\leqslant \sqrt{\epsilon}\sigma\sqrt{\dfrac{\pi}{2}}\label{eq47}
	\end{align}
	
	Denote the $\mathcal{H}_\infty$ be the filtration generated by the environment and the branching structure (without movement), then
	\begin{align*}
		&\P\left(M>n^{\frac{3}{4}}, \sup\limits_{t\in [0,1]}\log_+Z_{tT} \leqslant \sqrt{n}; T\leqslant \frac{n}{4}, Z_{\epsilon n}>0\right)
		\nonumber\\
		\leqslant
		&\P\left(
		\bigcup\limits_{k=1}^{\abs{\frac{n}{4}}} \left\{\exists \nu\in N_k, V(\nu)> n^{\frac{3}{4}} \right\}; \sup\limits_{t\in [0,1]}\log_+Z_{tT} \leqslant \sqrt{n}; T\leqslant \frac{n}{4}, Z_{\epsilon n}>0
		\right)
		\nonumber\\
		=
		&\P\left(
		\P\left(
		\bigcup\limits_{k=1}^{\abs{\frac{n}{4}}} \left\{\exists \nu\in N_k, V(\nu)> n^{\frac{3}{4}} \right\} |\mathcal{H}_\infty
		\right);
		\sup\limits_{t\in [0,1]}\log_+Z_{tT} \leqslant \sqrt{n}; T\leqslant \frac{n}{4}, Z_{\epsilon n}>0
		\right)
		\nonumber\\
		\leqslant
		&\P\left(
		\exp\left(\sqrt{n}\right)
		\sum\limits_{k=1}^{\abs{\frac{n}{4}}}\P\left( B_k>n^{\frac{3}{4}} \right);
		\sup\limits_{t\in [0,1]}\log_+Z_{tT} \leqslant \sqrt{n}; T\leqslant \frac{n}{4}, Z_{\epsilon n}>0
		\right)
		\nonumber\\
		\leqslant
		&\P\left(
		\exp\left(\sqrt{n}\right) \P\left( \exists 1\leqslant k\leqslant {\abs{\frac{n}{4}}}; B_k>n^{\frac{3}{4}} \right);\sup\limits_{t\in [0,1]}\log_+Z_{tT} \leqslant \sqrt{n}; T\leqslant \frac{n}{4}, Z_{\epsilon n}>0
		\right)
	\end{align*}
	Here $\left\{B_0:=0,B_1,\cdots,B_{{\abs{\frac{n}{4}}}}\right\}$ is a random walk with one step distribution $\ND(0,1)$, so using Brownian motion to embedding, and according to reflection principle, we know
	\[
	\P\left( \exists 1\leqslant k\leqslant {\abs{\frac{n}{4}}}; B_k>n^{\frac{3}{4}} \right)
	\leqslant
	2\P\left( B_1>2n^{\frac{1}{4}} \right)
	\leqslant
	n^{-\frac{1}{4}}\exp\left(-2\sqrt{n}\right).
	\]
	Together with \cref{Extinction Prob Decay}, we know that
	\begin{align}
		&\lim\limits_{n\rightarrow\infty}
		\P\left(M>n^{\frac{3}{4}}, \sup\limits_{t\in [0,1]}\log_+Z_{tT} \leqslant \sqrt{n}; T\leqslant \frac{n}{4}|Z_{\epsilon n}>0\right)
		\nonumber\\
		=&\dfrac{1}{K}\lim\limits_{n\rightarrow\infty}
		\sqrt{n\epsilon}\P\left(M>n^{\frac{3}{4}}, \sup\limits_{t\in [0,1]}\log_+Z_{tT} \leqslant \sqrt{n}; T\leqslant \frac{n}{4}; Z_{\epsilon n}>0\right)
		\nonumber\\
		\leqslant
		&\dfrac{1}{K}\lim\limits_{n\rightarrow\infty}
		\sqrt{\epsilon}
		n^{\frac{1}{4}}\exp\left(-\sqrt{n}\right)=0.\label{eq48}
	\end{align}
	Combine \cref{eq45,eq46,eq47,eq48}, we have
	\begin{align*}
		\limsup\limits_{n\rightarrow\infty}
		\P\left(M>n^{\frac{3}{4}}\left|\right.Z_{\epsilon n}>0\right)
		&\leqslant
		\epsilon\sigma\sqrt{\frac{\pi}{2}}+2\sqrt{\epsilon}.
	\end{align*}
	Therefore, combine \cref{Extinction Prob Decay},
	\begin{align}
		\lim\limits_{n\rightarrow\infty}
		\sqrt{n}\P\left(M>n^{\frac{3}{4}};Z_{\epsilon n}>0\right)
		&=
		\lim\limits_{n\rightarrow\infty}\sqrt{n}
		\P\left(Z_{\epsilon n}>0 \right)
		\P\left(M>n^{\frac{3}{4}}|Z_{\epsilon n}>0\right)
		\nonumber\\
		&\leqslant
		K\left(\sigma\sqrt{\dfrac{\pi}{2}}+2\right).\label{eq49}
	\end{align}
	For another item in \cref{eq44}, we know
	\begin{align}
		\P\left(M>n^{\frac{3}{4}};Z_{\epsilon n}=0\right)
		&\leqslant
		\P\left(M>n^{\frac{3}{4}};Z_{\epsilon n}=0; \sup_{1\leqslant k\leqslant n\epsilon} Z_k \leqslant \exp\left(\sigma\sqrt{n}\right)\right)
		\nonumber\\
		&+
		\P\left(Z_{\epsilon n}=0; \sup_{1\leqslant k\leqslant n\epsilon} Z_k > \exp\left(\sigma\sqrt{n}\right)\right)\label{eq50}
	\end{align}
	Using the filtration $\mathcal{H}_\infty$, with the same argument, we can obtain
	\begin{align}
		&\P\left(M>n^{\frac{3}{4}};Z_{\epsilon n}=0; \sup_{1\leqslant k\leqslant \abs{n\epsilon}} Z_n \leqslant \exp\left(\sigma\sqrt{n}\right)\right)
		\nonumber\\
		=
		&\P\left(
		\bigcup\limits_{k=1}^{\epsilon n} \left\{\exists \nu\in N_k, V(\nu)> n^{\frac{3}{4}} \right\};
		Z_{\epsilon n}=0; \sup_{1\leqslant k\leqslant \abs{n\epsilon}} Z_n \leqslant \exp\left(\sigma\sqrt{n}\right)\right)
		\nonumber\\
		\leqslant
		&\P\left(
		\exp\left(\sigma\sqrt{n}\right)\P\left(\exists 1\leqslant k\leqslant \epsilon n;B_k>n^{\frac{3}{4}} \right)
		;Z_{\epsilon n}=0; \sup_{1\leqslant k\leqslant \abs{n\epsilon}} Z_n \leqslant \exp\left(\sigma\sqrt{n}\right)\right)
		\nonumber\\
		\leqslant
		& \exp\left(\sigma\sqrt{n}\right) 2\sqrt{\epsilon}n^{-\frac{1}{4}}\exp\left(-\frac{\sqrt{n}}{2\epsilon}\right)
		\P\left(Z_{\epsilon n}=0; \sup_{1\leqslant k\leqslant \abs{n\epsilon}} Z_n \leqslant \exp\left(\sigma\sqrt{n}\right)\right)
		\nonumber\\
		\leqslant
		&\exp\left(\sigma\sqrt{n}\right) 2\sqrt{\epsilon}n^{-\frac{1}{4}}\exp\left(-\frac{\sqrt{n}}{2\epsilon}\right).\nonumber
	\end{align}
	Thus, for $\epsilon$ small enough, we know
	\begin{equation}\label{eq51}
		\lim\limits_{n\rightarrow\infty}\sqrt{n}\P\left(M>n^{\frac{3}{4}};Z_{\epsilon n}=0; \sup_{1\leqslant k\leqslant \abs{n\epsilon}} Z_n \leqslant \exp\left(\sigma\sqrt{n}\right)\right) = 0.
	\end{equation}
	About another item of \cref{eq50}, 
	define $ \rho:= \inf\{k:Z_k>\exp(\sigma\sqrt{n})\} $
	\begin{align}
		&\P\left(Z_{\epsilon n}=0; \sup_{1\leqslant k\leqslant n\epsilon} Z_n > \exp\left(\sigma\sqrt{n}\right)\right)
		\nonumber\\
		=&
		\sum_{k=1}^{\abs{n\epsilon}}
		\sum_{y=\abs{\exp\left(\sigma\sqrt{n}\right)}+1}^{\infty}
		\P\left(\rho=k,Z_k=y,Z_{n\epsilon}=0\right)
		\nonumber\\
		\leqslant&
		\sum_{k=1}^{\abs{n\epsilon}}
		\sum_{y=\abs{\exp\left(\sigma\sqrt{n}\right)}+1}^{\infty}
		\P\left(\rho=k,Z_k=y\right)
		\E\left(  \P_\xi\left(Z_{n\epsilon} = 0 \right)^k \right)
		\nonumber\\
		\leqslant&
		\P\left( \sup_{k\geqslant1} Z_n > \exp\left(\sigma\sqrt{n}  \right)\right)
		\E\left(
		\P_\xi\left(Z_{n\epsilon} = 0 \right)^{\abs{\exp(\sigma\sqrt{n})}}
		\right)\label{eq52}
	\end{align}
	
	\begin{lemma}\label{lemme of some extra}
		Under \Cref{A1}, for any $\epsilon>0$, and $x<0$,
		\begin{align*}
			\lim\limits_{n\rightarrow\infty}\P
			\left( \dfrac{1}{\sigma\sqrt{n}}\log\left(1-F_{0,\abs{n\epsilon}}
			\left(0\right)\right) <x \right)
			=
			\P\left(
			\inf\limits_{t\in[0,\epsilon]}B(t)<x \right).
		\end{align*}
		Here $\{B_t;t\geqslant 0\}$ is a standard Brownian motion.
	\end{lemma}
	
	\begin{proof}[Proof of \Cref{lemme of some extra}]
		For any $\epsilon>0$, under \Cref{A1}, by the moment estimate, we have
		\begin{equation*}
			\lim\limits_{n\rightarrow\infty}\P\left(
			\log\left(1+\sum_{k=1}^n\tilde{F}_k\right)>\exp\left(\sigma\epsilon\sqrt{n}\right) \right)
			=0
		\end{equation*}
		Together with \cref{Prob Upper,Prob lower}, the proof is over.
	\end{proof}
	
	\
	
	For any $x\in(0,1)$ and $y>0$, 
	the inequality $(1-x)^y\leqslant \exp(-xy)$ holds, then
	\begin{align*}
		P_{\xi}\left(Z_{n\epsilon}=0
		\right)^{\exp\left(\sigma\sqrt{n}\right)}
		&\leqslant
		\left(1-P_\xi\left(Z_{n\epsilon}>0\right)\right)^{\exp\left(\sigma\sqrt{n}\right)}
		\ind{P_\xi\left(Z_{n\epsilon}>0\right)\geqslant \exp\left(-\frac{\sigma\sqrt{n}}{2}\right)}
		\nonumber\\
		&+
		\ind{P_\xi\left(Z_{n\epsilon}>0\right)<
			\exp\left(-\frac{\sigma\sqrt{n}}{2}\right)}.
		\nonumber\\
		&\leqslant
		\exp\left(-\exp\left(\dfrac{\sigma\sqrt{n}}{2}\right)\right)
		\ind{P_\xi\left(Z_{n\epsilon}>0\right)\geqslant \exp\left(-\frac{\sigma\sqrt{n}}{2}\right)}
		\nonumber\\
		&+
		\ind{P_\xi\left(Z_{n\epsilon}>0\right)< \exp\left(-\frac{\sigma\sqrt{n}}{2}\right)}
	\end{align*}
	Take exception in both side, then
	\begin{align*}
		\E\left(P_{\xi}\left(Z_{n\epsilon}=0
		\right)^{\exp\left(\sigma\sqrt{n}\right)}\right)
		&\leqslant
		\exp\left(-\exp\left(\dfrac{\sigma\sqrt{n}}{2}\right)\right)
		\P
		\left( \dfrac{1}{\sigma\sqrt{n}}\log\left(1-F_{0,n\epsilon}
		\left(0\right)\right) \geqslant -\dfrac{1}{2} \right)
		\nonumber\\
		&+
		\P
		\left( \dfrac{1}{\sigma\sqrt{n}}\log\left(1-F_{0,n\epsilon}
		\left(0\right)\right) < -\dfrac{1}{2} \right)
	\end{align*}
	Let $n\rightarrow\infty$, then
	\begin{align*}
		\limsup\limits_{n\rightarrow\infty}
		\E\left(P_{\xi}\left(Z_{n\epsilon}=0
		\right)^{\exp\left(\sigma\sqrt{n}\right)}\right)
		&\leqslant
		\P\left(
		\inf\limits_{t\in[0,\epsilon]}B(t)<-\dfrac{1}{2} \right).
	\end{align*}
	Next take $\epsilon\rightarrow0$, and note that
	\begin{equation*}
		\lim\limits_{\epsilon\rightarrow0}\inf\limits_{t\in[0,\epsilon]}B(t)
		=0,\quad \P~\text{almost surely.}
	\end{equation*}
	So, we know
	\begin{align}\label{eq53}
		\lim\limits_{\epsilon\rightarrow0}
		\limsup\limits_{n\rightarrow\infty}
		\E\left(P_{\xi}\left(Z_{n\epsilon}=0
		\right)^{\exp\left(\sigma\sqrt{n}\right)}\right)
		=0.
	\end{align}
	
	What's more, use \cref{eq53} and repeat the proof in \cite{Afanasyev1999} (the Lemma 4 is replaced by \cref{eq53}, and other arguments are consistent), we can prove that there exist some finite positive constant $K_2$, such that
	\begin{equation}\label{eq54}
		\lim\limits_{n\rightarrow\infty}
		\sqrt{n}\P\left( \sup_{k\geqslant1} Z_n > \exp\left(\sigma\sqrt{n}  \right)\right) = K_2.
	\end{equation}
	Combine \cref{eq52,eq53,eq54}, we know
	\begin{equation}\label{eq55}
		\lim\limits_{\epsilon\rightarrow0}\limsup\limits_{n\rightarrow\infty}\sqrt{n}\P\left(Z_{\epsilon n}=0; \sup_{1\leqslant k\leqslant n\epsilon} Z_n > \exp\left(\sigma\sqrt{n}\right)\right) = 0.
	\end{equation}
	Therefore, together with \cref{eq50,eq51,eq55}, we have
	\begin{align}\label{eq56}
		\lim\limits_{\epsilon\rightarrow0}\limsup\limits_{n\rightarrow\infty}
		\sqrt{n}
		\P\left(M>n^{\frac{3}{4}};Z_{\epsilon n}=0\right)=0.
	\end{align}
	Combine \cref{eq44,eq49,eq56}, just take $C_2:= K\left(\sigma\sqrt{\dfrac{\pi}{2}}+2\right) $, thus
	\begin{align}
		\limsup\limits_{n\rightarrow\infty}\sqrt{n}\P\left(M>n^{\frac{3}{4}} \right)
		\leqslant C_2.
	\end{align}
	Hence, the right hand of \cref{tail estimate} is proved.
	
	Now, the proof of \Cref{Tail prob of M} is finished.

	\bibliography{sn-bibliography}
	\bibliographystyle{plain}
	
\end{document}